\documentclass[review,onefignum,onetabnum]{siamart171218}
\usepackage{amssymb,amsmath}
\usepackage{lineno}
\usepackage{wrapfig}
\usepackage{algorithmic}
\usepackage{xspace}
\usepackage{subfigure}
\usepackage{extarrows}
\usepackage{stmaryrd}
\usepackage{footmisc}
\newcommand{\xx}{{\bf x}}
\newcommand{\vv}{{\bf v}}
\newcommand{\yy}{{\bf y}}

\newcommand{\ab}{{\bf a}}
\newcommand{\bb}{{\bf b}}
\newcommand{\cc}{{\bf c}}

\newcommand{\A}{{\bf A}}
\newcommand{\X}{{\bf X}}

\newcommand{\I}{{\bf I}}
\newcommand{\B}{{\bf B}}
\newcommand{\C}{{\bf C}}

\newcommand{\Z}{{\bf Z}}

\newcommand{\qq}{{\bf q}}
\newcommand{\xt}{{\bf x}}

\newcommand{\MATLAB}{\textsc{Matlab}\xspace}
\newcommand\norm[1]{\left\lVert#1\right\rVert}

\nolinenumbers

\usepackage{lipsum}
\usepackage{amsfonts}
\usepackage{graphicx}
\usepackage{epstopdf}
\usepackage{algorithmic}
\ifpdf
  \DeclareGraphicsExtensions{.eps,.pdf,.png,.jpg}
\else
  \DeclareGraphicsExtensions{.eps}
\fi

\newsiamremark{remark}{Remark}
\newsiamremark{hypothesis}{Hypothesis}
\crefname{hypothesis}{Hypothesis}{Hypotheses}
\newsiamthm{claim}{Claim}

\headers{Nonnegative Tensor Decomposition Via
CNO}{}

\title{Nonnegative Tensor Decomposition Via
Collaborative Neurodynamic Optimization
}

\author{Salman Ahmadi-Asl\thanks{Lab of Machine Learning and Knowledge Representation, Innopolis University, Innopolis, Russia, and Skolkovo Institute of Science and Technology, Center for Artificial Intelligence Technology, Moscow, Russia,
(\email{s.ahmadiasl@innopolis.ru}).}
\and Valentin Leplat\thanks{Innopolis University, Innopolis, Russia, (\email{v.leplat@innopolis.ru}), ``The work of V.L was done with Skoltech, finalized at Innopolis''.}  \and Anh-Huy Phan\thanks{Skolkovo Institute of Science and Technology, Center for Artificial Intelligence Technology, Moscow, Russia, (\email{a.phan@skoltech.ru}).} \and Andrzej Cichocki\thanks{%Systems Research Institute of Polish Academy of Science, Warsaw, Poland 
Skolkovo Institute of Science and Technology, Center for Artificial Intelligence Technology, Moscow, Russia and Systems Research Institute of Polish Academy of Science, Warsaw, Poland
(\email{cichockiand@gmail.com}).} 
}

\usepackage{amsopn}

\makeatletter
\newcommand*{\addFileDependency}[1]{% argument=file name and extension
  \typeout{(#1)}% latexmk will find this if $recorder=0 (however, in that case, it will ignore #1 if it is a .aux or .pdf file etc and it exists! if it doesn't exist, it will appear in the list of dependents regardless)
  \@addtofilelist{#1}% if you want it to appear in \listfiles, not really necessary and latexmk doesn't use this
  \IfFileExists{#1}{}{\typeout{No file #1.}}% latexmk will find this message if #1 doesn't exist (yet)
}
\makeatother

\newcommand*{\myexternaldocument}[1]{%
    \externaldocument{#1}%
    \addFileDependency{#1.tex}%
    \addFileDependency{#1.aux}%
}
\ifpdf
\hypersetup{
  pdftitle={Nonnegative Tensor Decomposition Via
Collaborative Neurodynamic Optimization},
  pdfauthor={}
}
\fi

\myexternaldocument{ex_supplement}

\begin{document}

\maketitle

% REQUIRED
\begin{abstract}
This paper introduces a novel collaborative neurodynamic model for computing nonnegative Canonical Polyadic Decomposition (CPD). The model relies on a system of recurrent neural networks to solve the underlying nonconvex optimization problem associated with nonnegative CPD. Additionally, a discrete-time version of the continuous neural network is developed. To enhance the chances of reaching a potential global minimum, the recurrent neural networks are allowed to communicate and exchange information through particle swarm optimization (PSO). Convergence and stability analyses of both the continuous and discrete neurodynamic models are thoroughly examined. Experimental evaluations are conducted on random and real-world datasets to demonstrate the effectiveness of the proposed approach.
\end{abstract}

% REQUIRED

% REQUIRED

\section{Introduction}
Tensors are a generalization of matrices and vectors, and tensor decompositions and tensor networks are utilized to represent tensors in more concise forms. Unlike the matrix rank, the notion of rank for tensors is not unique and several tensor ranks and corresponding decompositions have been introduced during the past decades. These decompositions include the Canonical Polyadic Decomposition (CPD) \cite{hitchcock1927expression,hitchcock1928multiple}, Tucker decomposition \cite{tucker1964extension}, and Tensor Train decomposition \cite{oseledets2011tensor}, among others. The CPD is particularly significant and widely used in various machine learning and data science tasks. It decomposes a higher-order tensor into a sum of rank-1 tensors \cite{hitchcock1927expression,hitchcock1928multiple} with the minimum number of terms. If the number of rank-1 terms is not minimum, it is referred to as Polyadic Decomposition (PD). The CPD has found applications in compression \cite{phan2020stable,li2010tensor}, completion \cite{asante2021matrix,yokota2016smooth}, Hammerstein identification \cite{elden2019solving}, and clustering \cite{alexandrov2019nonnegative,wei2019multiscale}. While the best rank CPD approximation of a tensor may not exist in general \cite{de2008tensor}, the non-negative CPD problem is always well-posed and has a global solution \cite{lim2009nonnegative}. Hierarchical Alternating Least Squares (HALS) \cite{cichocki2009fast,cichocki2007hierarchical,cichocki2009nonnegative} and ALS are well-known iterative algorithms for decomposing a tensor in the CPD format. HALS has been extensively used to compute a nonnegative CPD, and an extension of HALS was proposed in \cite{phan2011extended} for the computation of nonnegative Tucker decomposition \cite{tucker1964extension,tucker1966some}. Additionally, a generalized form of the Multiplicative Updating Rules (MUR) \cite{lee2000algorithms} has been extended from matrices to tensors for the nonnegative CPD model. For more details, refer to \cite{cichocki2009nonnegative} and the references therein.

Collaborate neurodynamic optimization has shown to be a powerful approach for solving convex and nonconvex optimization problems. Such applications include computing the LU and the Cholesky decompositions \cite{wang1993recurrentLU}, sparse signal recovery \cite{che2022sparse}, nonnegative matrix factorization \cite{che2018nonnegative,fan2016collective}, Boolean matrix factorization \cite{li2022boolean}, solving linear matrix equation \cite{wang1993recurrent,zhang2002recurrent}, computing the Moore-Penrose pseudoinverse of matrices \cite{wang1997recurrent} etc. A collective neurodynamic model based on the HALS algorithm is proposed in \cite{fan2017collective}. However, the paper is short and the efficiency of the algorithm is not extensively discussed. Here, we expand that paper by providing extensive experiments and formulate the collaborative neurodynamic model based on the CP-ALS algorithm. To the best of our knowledge, this is the first expanded paper that  discusses the possibility of using  collaborative neurodynamic optimization for tensor decomposition.     

Our key contributions are as follows:

\begin{itemize}
    \item Proposing a collaborative neurodynamic model for computation of the CPD and its constrained form.
    
    \item Proposing a discrete-time projection neural network for the nonnegative CPD.

    \item {Developing accelerated versions of the continuous and discrete neurodynamic models by the Hessian preconditioning strategy.}  
    
    \item The extensive experimental results and studying the proposed approach numerically. Our algorithm provides better results for data tensors with high collinearity.
\end{itemize}

%\lipsum[2-3]

The organization of the rest of the paper is as follows. In Section \ref{Pre:Sec}, we present the necessary notations and definitions we need in the rest of the paper. The problem of nonnegative CPD problem is formulated as a collaborative neurodynamic optimization problem in section \ref{Sec:Pro} and its convergence analysis is studied in Section \ref{Con:CNO}. A discrete version of the developed continues CPD is also developed in Section \ref{Sec:DTPNN} and its stability is shown in Section \ref{Sec:stab}. The experimental results are provided in section \ref{Sec:Sim} and finally, a conclusion is given in Section \ref{Sec:Con}. 

\section{Preliminaries}\label{Pre:Sec}
Let us introduce the basic concepts, notations and definitions which we need in the paper. Tensor, matrices and vectors are represented by underlined bold capital letters, bold capital letters and bold lower letters. The Hadamard product, Khatri–Rao product, Kronecker product and outer product are denoted by ``$\ast$'', ``$\odot$'', ``$\otimes$'' and ``$\circ$'', respectively. 

{For a function $f$ defined on domain $D$, a point $x^*\in D$ is called a partial optimum if there exists a neighborhood $N$ of $x^*$, such that $f(x^*)\leq f(x)$, for all $x\in D\cap N$.}

The notation $\times_n$ represents mode-$n$ product of an $N$-th order tensor and a matrix. It is a generalization of matrix-matrix multiplication, defined for $\underline{\mathbf X}\in\mathbb{R}^{I_1\times I_2\times \cdots\times I_N}$ and ${\mathbf B}\in\mathbb{R}^{J\times I_n}$, as
\[
{\left( {{\underline{\mathbf X}}{ \times _n}{\mathbf B}} \right)_{{i_1},\ldots {,i_{n - 1}},j,{i_{n + 1}}, \ldots, {i_N}}} = \sum\limits_{{i_n=1}}^{{I_N}} {{x_{{i_1},{i_2}, \ldots ,{i_N}}}{b_{j,{i_n}}}},
\]
for $j=1,2,\ldots,J$. Here, we have $
\underline{\mathbf X}\times_{n}{\mathbf B}\in\mathbb{R}^{I_1\times \cdots \times I_{n-1}\times J\times I_{n+1}\times \cdots \times I_N}.$
The CPD represents a tensor as a sum of rank one tensors. The CPD of an $N$th-order tensor, $ \underline{\bf X} \in \mathbb{R}^{I_1\times I_2\times\cdots\times I_N}$ is represented as follows \begin{align*}
     \underline{\bf X} &\cong \sum_{r = 1}^R {\bf a}^{(1)}_r \circ {\bf a}^{(2)}_r \circ \cdots \circ {\bf a}^{(3)}_r = \underline{\bf I} \times_1 {\bf A}^{(1)} \times_2 {\bf A}^{(2)} \cdots\times_N {\bf A}^{(N)}\\
      &= \llbracket \underline{\bf I}; {\bf A}^{(1)}, {\bf A}^{(2)},\ldots,{\bf A}^{(N)}\rrbracket,
\end{align*}
where $\underline{\bf I}$ is a super-diagonal core tensor whose all diagonal elements are equal to 1, and ${\bf A}^{(n)} = [{\bf a}_1^{(n)}, {\bf a}_2^{(n)},...,{\bf a}_R^{(n)}] \in \mathbb{R}^{I_n \times R},\, n = 1,2,\ldots,N,$ are factor matrices. For simplicity of presentation, we can consider the so-called Kruskal format $\llbracket {\bf A}^{(1)}, {\bf A}^{(2)},\ldots,{\bf A}^{(N)}\rrbracket$.

The nonnegative CPD of a tensor $\X$ is mathematically  formulated as follow
\begin{eqnarray}\label{NonOPT}
\min_{{\bf A}^{(1)}\geq 0,\ldots\,,{\bf A}^{(N)}\geq 0}\|\underline{\bf X}-\llbracket {\bf A}^{(1)}, {\bf A}^{(2)},\ldots,{\bf A}^{(N)}\rrbracket\|^2_F.
\end{eqnarray}
The problem \eqref{NonOPT} is a nonconvex optimization problem with respect to all factor matrices, which means that finding the global minimum is generally hard. However, the problem is convex with respect to a single factor matrix, which means that we can use convex optimization techniques to update one factor matrix at a time, while fixing the others.

The problem \eqref{NonOPT} minimises the Frobenius norm  as the measure of approximation error. However, other norms or divergences can also be used, depending on the nature and distribution of the data such as Kullback–Leibler (K–L) divergence \cite{lee2000algorithms} and Alpha-Beta divergences \cite{chernoff1952measure,eguchi2001robustifying,mihoko2002robust}. Besides, it is also possible to regularize the objective function in \eqref{NonOPT} using for example graph regularization \cite{qiu2019graph} which is used in some applications. We can impose unit simplex constraints on the factor matrices in the problem \eqref{NonOPT}. This constraint can be useful for some applications, such as topic modeling, clustering, or classification, where the factor matrices can be interpreted as probability distributions or membership indicators. The unit simplex constraint can also improve the stability and efficiency of the matrix factorization, as it reduces the ambiguity and redundancy of the solution. 

\begin{figure}
    \begin{center}
    \includegraphics[width=0.6\linewidth]{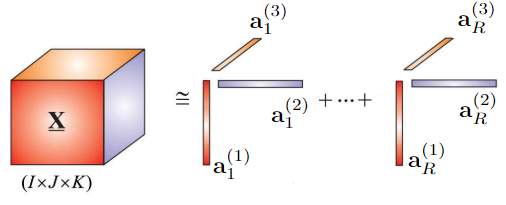}
    \caption{Polyadic decomposition of tensor $\X$ of Rank-$R$ \cite{cichocki2009nonnegative}}\label{fig:CP}.
    \end{center}
\end{figure}

To introduce the collaborative neurodynamic models, we consider the following constrained optimization problem
\begin{eqnarray}\label{OP_1}
\min_{\xx\in\Omega} f(\xx)
\end{eqnarray}
which seeks a vector ${\xx}\in\mathbb{R}^{N}$ over a feasible region $\Omega$. The so-called one layer Recurrent Neural Network (RNN) associated to  optimization problem \eqref{OP_1} is defined as follows
\begin{eqnarray}\label{RNN_1}
\epsilon \frac{d{\xx}(t)}{dt}=-{\xx}(t)+P_{\Omega}({\xx}(t)-\nabla f({\xx}(t))),
\end{eqnarray}
where $\epsilon$ is a small positive quantity (time constant), ${\xx}$ is the state vector, $P_{\Omega}$ is the activation function and $\nabla f$ is the gradient of the objective function $f$. The activation function is defined as a projection of $y$ onto the feasible set $\Omega$
\begin{eqnarray}
P_{\Omega}({\xx})=\arg\min_{\yy\in\Omega}\|\xx-\yy\|_2.
\end{eqnarray}
For example, the activation function for the box constraints is given by
\begin{eqnarray}\label{acfunc}
P_{\Omega}(x)=\left\{\begin{aligned}
u,&&x>u,\\
x,&& x\in [l,u],\\
l,\,&& x<l,
\end{aligned}\right.
\end{eqnarray}
For the special case of $u_i=+\infty,\,l_i=0$, we have the 
so-called rectified linear unit defined as follows 
\[
P_{\Omega}(x)= [x]_{+} = \left\{\begin{aligned}
0,&&x<0,\\
x,&&{x}\geq 0.
\end{aligned}\right.
\]
The vector activation function   $P_{\Omega}:\mathbb{R}^n\rightarrow\Omega$ 
is defined similarly 
where the projection operator is applied to all elements (element-wise projection). It is known  \cite{xia2000stability} that if the objective function $f$ in \eqref{OP_1} is convex, then the solution to the RNN \eqref{RNN_1} over $[0,t]$ for a time-step $t$ is globally stable and convergent to the optimal solution of problem \eqref{OP_1}. This result was generalized to pseudoconvex optimization and nonconvex optimization problems in \cite{hu2006solving} and \cite{li2015one}, respectively. The collaborative neurodynamic optimization uses a set of RNNs that communicate with each other and share their information for searching a global minimum. 
\begin{figure}
\begin{center}
\includegraphics[width=.6\linewidth]{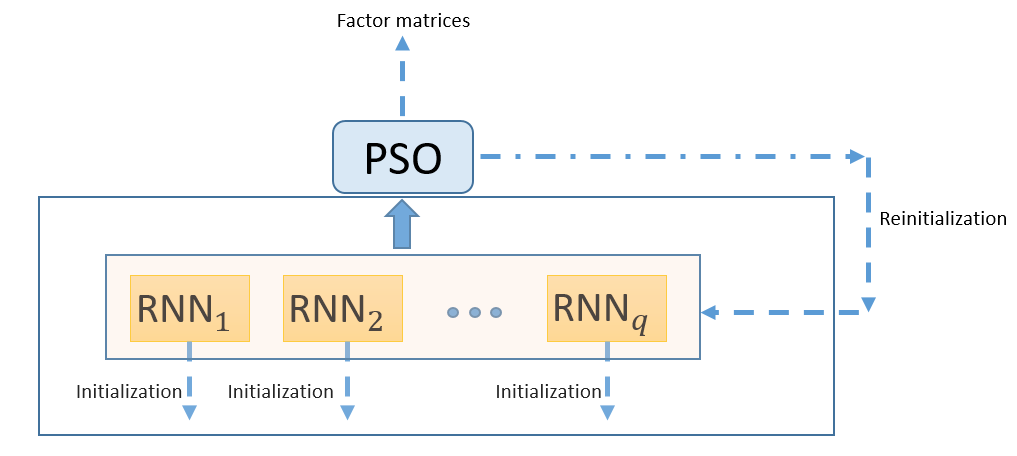}\\
\caption{\small{The procedure of collaborative neurodynamic optimization for nonnegative CPD.}}\label{Pic2}
\end{center}
\end{figure}

Particle swarm optimization (PSO) is an effective population-based optimization method inspired by the collective behavior of swarms of insects or fishes. Its purpose is to find the global optimum of a constrained optimization problem \cite{kennedy1995particle}. In PSO, a group of $N$ particles explores the feasible region and communicates with each other to exchange information about the best locations. The update process for each particle can be described as follows:
\begin{eqnarray}\label{PSO}
\vv_n^{(k+1)}&=&\alpha\vv_n^{(k)}+\beta_1\gamma_1({\bf p}^{(k)}_n-\xx_n^{(k)})+\beta_2\gamma_2({\bf p}^{(k)}_{best}-\xx_n^{(k)}),\\\label{PSO2}
\xx_n^{(k+1)}&=&\xx_n^{(k)}+\vv_n^{(k+1)}\hspace*{5cm},
\end{eqnarray}
where $\alpha\in[0,1]$ represents the inertia weight that controls the momentum of the particles. The variables $\gamma_1$ and $\gamma_2$ are two uniform random variables in the range $[0,1]$, introducing stochasticity to the search process. The acceleration constants $\beta_1$ and $\beta_2$ determine the influence of the individual and global bests on the particles.  In this context, $p^{(k)}_n$ refers to the best location of the $n$-th particle, while $p_{\text{best}}^{(k)}$ represents the best location among all particles at iteration $k$. The particles are moved along their respective velocities $\vv_n^{(k)}$. The best position of each particle and the best global solution of all particles are updated at each iteration as follows
\begin{align}\label{localsearch}
{\bf p}^{(k+1)}_{n}&= \begin{cases}
   {\bf p}_n^{(k)},\quad & f(\xx_n^{(k+1)})\geq f({\bf p}_n^{(k+1)})\\
\xx_n^{(k+1)},\quad &  f(\xx_n^{(k+1)})< f({\bf p}_n^{(k+1)})\hspace*{-.2cm}
\end{cases}\\
\label{globalsearch}
{\bf p}_{best}^{(k+1)}&=\arg\min_{{\bf p}_i^{(k+1)}} f(p_i^{(k+1)}),\,\quad 1\leq i\leq q.
\end{align}
The RNNs are reinitialized according to the PSO formula \eqref{PSO}-\eqref{PSO2} and using the updating rules \eqref{localsearch} and \eqref{globalsearch}. 

For the tensor decomposition studied in this paper, we define the objective function $f$ as squares of the Frobenius norm of the approximation error 
\begin{eqnarray}\label{objec_func}
F({\bf A}^{(1)},{\bf A}^{(2)},\ldots,{\bf A}^{(N)})=\|\underline{\bf X}-\llbracket {\bf A}^{(1)}, {\bf A}^{(2)},\ldots,{\bf A}^{(N)}\rrbracket\|^2_F.
\end{eqnarray}
The next quantity prevents the particles from converging to a local minimum before reaching a global solution. It measures the distance between the best position of each particle and the best global position of all particles in the swarm.
\begin{definition}
The diversity of particles in the swarm is measured via the following formula
\[
DI(k)=\frac{1}{q}\sum_{n=1}^{q}\|{{\bf p}_n^{(k)}-{\bf p}_{best}^{(k)}}\|_2,
\]
where $q$ is the number of RNNs in the swarm. 
\end{definition}
A small diversity means that the particles are concentrated on a small region of the feasible region and there is a risk of convergence of all particles to a local minimum. In \cite{yan2016collective}, it was proven that for sufficient large diversity, the collaborative neurodynamic models almost surely converge to a global solution. To enhance the search procedure and allow the particles to look for a global minimum, the wavelet mutation is an effective approach used extensively in the literature \cite{che2022sparse,fan2016collective}.
\begin{align}\label{Wavelet}
\xx^{(k+1)}_{n}= \begin{cases}
\xx_n^{(k)}+\kappa(u_n-\xx_n^{(k)}),\, &\kappa>1\\
\xx_n^{(k)}+\kappa(\xx_n^{(k)}-l_n),\, & \kappa<1 \end{cases}
\end{align}
where $u_n$ and $l_n$ are the upper and lower bounds of the state $\xx_n$, respectively and $\kappa(\phi)$ is the Gabor wavelet function defined as
\begin{equation}\label{Wavelet2}
\kappa(\phi)=\frac{1}{\sqrt{a}}\exp(-\frac{\phi}{2a})\cos(\frac{5\phi}{a}),
\end{equation}
where $a=\exp(10(k/k_{max}))$, $k$ is the current iteration number, $k_{max}$ is the maximum number of iterations, $\phi\in [-2.5a,2.5a]$. A small $|\kappa|$ leads to a small search space while a larger $\kappa$ provides a larger searching space. For $\kappa \approx 1$, $\xx_n^{(k+1)}$ approaches its upper bound $u_n$, while for $\kappa \approx -1$, it approaches its lower bound $l_n$.
\section{CNO-based algorithm for nonnegative CPD (CNO-CPD)}\label{Sec:Pro}
For the simplicity of presentation, we describe the proposed collaborative neurodynamic approach for third-order tensors while the generalization to higher-order tensors is straightforward. 
To compute a rank-$R$ CPD of a tensor $\underline{\bf X}\in\mathbb{R}^{I_1\times I_2\times I_3}$, we consider the objective function $F$ defined as follows:
\begin{eqnarray}\label{F_defin}
\nonumber
F({\bf A},{\bf B},{\bf C})&=&\frac{1}{2}\|\underline{\bf X}-[[{\bf A},{\bf B},{\bf C}]]\|_F^2=
\frac{1}{2}\|\underline{\bf X}\|_F^2+\frac{1}{2}\|[[{\bf A},{\bf B},{\bf C}]]\|_F^2-\langle\underline{\bf X},[[{\bf A},{\bf B},{\bf C}]]\rangle\\&=&\label{object}
\frac{1}{2}\|\underline{\bf X}\|_F^2+\frac{1}{2}\|\A(\C\odot\B)^T\|_F^2-\langle \X_{(1)},\A(\C\odot\B)^T\rangle,
\end{eqnarray}
where ${\bf A}\in\mathbb{R}^{I_1\times R},\,{\bf B}\in\mathbb{R}^{I_2\times R}$ and ${\bf C}\in\mathbb{R}^{I_3\times R}$. The gradients of $F$ with respect to three factor matrices are given by
\begin{align}
\nabla_{\A} F&=\A(\C\odot\B)^T(\C\odot\B)-\X_{(1)}(\C\odot\B)      \notag 
\\
&= \A\left((\C^{T}\C) \ast (\B^T\B)\right)-\X_{(1)}(\C\odot\B) \label{EQ:1}\\
\nabla_{\B} F&=\B\left((\C^{T}\C)  \ast (\A^T\A)\right)-\X_{(2)}(\C\odot\A) \label{EQ:2}\\%\B(\C\odot\A)^T(\C\odot\A)-\X_{(2)}(\C\odot\A),\\
\nabla_{\C} F&=\C\left((\B^{T}\B) \ast(\A^T\A)\right)-\X_{(3)}(\B\odot\A) \label{EQ:3} %\C(\B\odot\A)^T(\B\odot\A)-\X_{(3)}(\B\odot\A).
\end{align}
We write product of two Khatri-Rao products as Hadamard product to reduce the computational complexity $(\C\odot\B)^T(\C\odot\B)=(\C^{T}\C) \ast (\B^T\B)$ in \eqref{EQ:1}-\eqref{EQ:3} \cite{kolda2009tensor}. 
Now, the neurodynamic formulation of the CPD is written as follows
\begin{align}\label{CON_1}
\epsilon_1\frac{d{\bf a}}{dt}&=-{\bf a} + [{\bf a}-{H_{\bf a}}^{-1} \nabla_{{\bf a}}F({\bf a},{\bf b},{\bf c})]_{+},\\
\epsilon_2\frac{d{\bf b}}{dt}&=-{\bf b} + [{\bf b}-{H_{\bf b}}^{-1}  \nabla_{{\bf b}}F({\bf a},{\bf b},{\bf c})]_{+},\\\label{CON_33}
\epsilon_3\frac{d{\bf c}}{dt}&=-{\bf c} + [{\bf c}-{H_{\bf c}}^{-1}  \nabla_{{\bf c}}F({\bf a},{\bf b},{\bf c})]_{+}
\end{align}
where ${\bf a}={\rm vec}(\A),\,{\bf b}={\rm vec}(\B)$, ${\bf c}={\rm vec}(\C)$, and $H_{\bf a}$, $H_{\bf b}$ and $H_{\bf c}$ are the Hessian matrices of $F()$ w.r.t. ${\bf a}$, ${\bf b}$ and ${\bf c}$.

For CPD, the Hessian can be  computed as
$H_{\bf a} = {\bf P}_{\A} \otimes \I_{I_1}$, ${\bf P}_{\A} = (\C^T \C) \ast (\B^T \B)$, {where $\I_{I_1}$ is an identity matrix of size $I_1\times I_1$}. %The rectifier $[a]_+={\rm max}(0,a)$.
The above neurodynamic system can be rewritten in the matrix form for factor matrices as 
\begin{align}\label{EQ:4}
\epsilon_1\frac{d{\A}}{dt}&=-\A+[\A-\nabla_{\A}F(\A,\B,\C){{\bf P}_{\A}^{-1}}]_{+},\\\label{EQ:5}
\epsilon_2\frac{d{\B}}{dt}&=-\B+[\B-\nabla_{\B}F(\A,\B,\C){{\bf P}_{\B}^{-1}}]_{+},\\\label{EQ:6}
\epsilon_3\frac{d{\C}}{dt}&=-\C+[\C-\nabla_{\C}F(\A,\B,\C){{\bf P}_{\C}^{-1}}]_{+}.
\end{align}
We can interpret ${\bf P}_{\Z}$ as symmetric preconditioner for the gradient of $F$ w.r.t. $\Z$ built at the current point. It is well known in optimization that preconditioning often leads to faster convergence. The proposed algorithm demonstrates similar observation. The used preconditioners require at least quadratic memory space $\mathcal{O}(R^2)$ and $\mathcal{O}(R^3)$ computational effort to invert it before being applied to the gradient. 
  
The three-timescale neurodynamic optimization \eqref{CON_1}-\eqref{CON_33} is a one-layer RNN which looks for a local minimum and finds all factor matrices all at once. However, it is possible to consider several RNNs using different initial points and $\epsilon_i,\,i=1,2,3$ to reach various local minima. Then, using the PSO method, the local minima can communicate with each other to approach a better local minimum. The more RNNs we use, the higher our chance of achieving a global minimum. Note that each neurodynamic model is independent of the others. So they can solve in a distributed and parallel manner, which allows for real-time optimization. The {CNO-CPD} is based on the CP-ALS algorithm, but the CNO algorithm can be formulated based on the HALS algorithm to compute a nonnegative tensor decomposition. More precisely, we can reformulate the minimization in \eqref{NonOPT} as a sequence of nonnegative rank-1 tensor approximations. To this end, consider the following minimization problem
\begin{eqnarray}\label{HALS}
\min_{{\bf a}^{(i)}\geq 0,\,{\bf b}^{(i)}\geq 0,\,{\bf c}^{(i)}\geq 0}\frac{1}{2}\|\underline{\X}^{(i-1)}-[[{\bf a}^{(i)},{\bf b}^{(i)},{\bf c}^{(i)}]]\|_F^2,
\end{eqnarray}
where
\begin{eqnarray}
\nonumber
\A^{(-i)}&=&[{\bf a}^{(1)},\ldots,{\bf a}^{(i-1)},{\bf a}^{(i+1)},\ldots,{\bf a}^{(R)}]\in\mathbb{R}^{I_1\times(R-1)},\\
\nonumber
\B^{(-i)}&=&[{\bf b}^{(1)},\ldots,{\bf b}^{(i-1)},{\bf b}^{(i+1)},\ldots,{\bf b}^{(R)}]\in\mathbb{R}^{I_2\times(R-1)},\\
\nonumber
\C^{(-i)}&=&[{\bf c}^{(1)},\ldots,{\bf c}^{(i-1)}, {\bf c}^{(i+1)},\ldots,{\bf c}^{(R)}]\in\mathbb{R}^{I_3\times(R-1)},\\
\nonumber
\underline{\X}^{(i-1)}&=&\underline{\X}-[[{\bf A}^{(-i)},{\bf B}^{(-i)},{\bf C}^{(-i)}]]\in\mathbb{R}^{I_1\times I_2\times I_3}.
\end{eqnarray}
 Here, the matrices $\A^{(-i)},\,\B^{(-i)}$ and $\C^{(-i)}$ are initialized and we update the vectors ${\bf a}^{(i)},\,{\bf b}^{(i)}$ and ${\bf c}^{(i)}$ by solving minimization problem \eqref{HALS} which is a special case of problem \eqref{NonOPT} (rank one tensor), but is easier to be handled. 

 The continuous CNO-based algorithm for computation of the CPD (CNO-CPD) is presented in Algorithm \ref{ALG:TSVDP} and Figure \ref{flo}.

 \begin{wrapfigure}{r}{0.4\textwidth}
 \begin{center}
\includegraphics[width=.4\columnwidth]{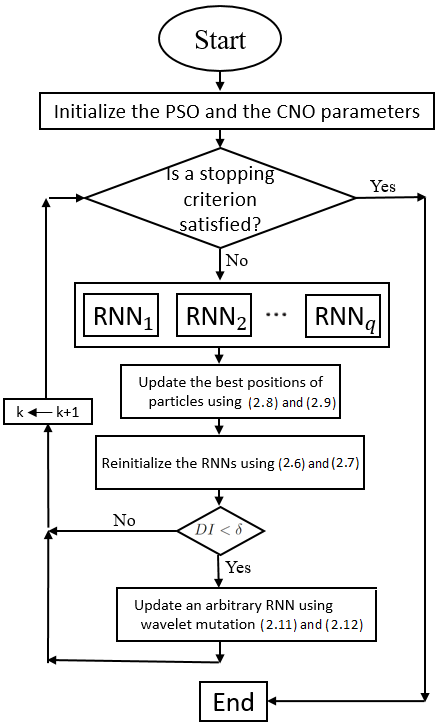}
\caption{\small{Flowchart of the proposed {CNO-CPD} method.}}\label{flo}
 \end{center}
\end{wrapfigure}
\subsection{Log barrier formulation}
The projection onto nonnegative orthant in the neurodynamic system \eqref{CON_1}-\eqref{CON_33} can be removed if we convert the constrained minimization problem \eqref{NonOPT} to an unconstrained one, e.g. using the log barrier method \cite{wright2006numerical}. Utilizing the benefits of the log barrier approach, such as its flexibility, efficiency, assured convergence, high accuracy, and ease of implementation, are the main motivations behind this conversion. To be more precise, we consider the minimization problem
\begin{eqnarray}\label{op_obje}
\min_{{\bf A},{\bf B},{\bf C}}  \tilde{F}({\bf A},{\bf B},{\bf C}).   
\end{eqnarray}
where $\tilde{F}({\bf A},{\bf B},{\bf C})=F({\bf A},{\bf B},{\bf C})+\gamma\,G({\bf A},{\bf B},{\bf C})$, $F$ was defined in \eqref{F_defin}, $\gamma$ is a regularization parameter and $G= \sum_{i,r} \log({\bf A}(i,r))+ \sum_{j,r} \log({\bf B}(j,r))+ \sum_{k,r} \log({\bf C}(k,r))$. Then, by some straightforward computations, the neurodynamic system \eqref{CON_1}-\eqref{CON_33} is converted to 
\begin{eqnarray}\label{neuro_logbar}
%\left\{\begin{aligned}
\epsilon_1\frac{d{\bf a}}{dt}=-{\tilde{H}_{\bf a}}^{-1}\nabla_{{\bf a}}\tilde{F}({\bf a},{\bf b},{\bf c}),\\\label{neuro_logbar2}
\epsilon_2\frac{d{\bf b}}{dt}=-{\tilde{H}_{\bf b}}^{-1}\nabla_{{\bf b}}\tilde{F}({\bf a},{\bf b},{\bf c}),\\\label{neuro_logbar3}
\epsilon_3\frac{d{\bf c}}{dt}= -{\tilde{H}_{\bf c}}^{-1}\nabla_{{\bf c}}\tilde{F}({\bf a},{\bf b},{\bf c}),
%\end{aligned}\right.
\end{eqnarray}
where $\nabla_{{\bf a}}\tilde{F}=\nabla_{{\bf a}}{F}+\gamma\,\nabla_{{\bf a}}{G}$ is the gradient of the objective function with respect to the first factor matrix, $\nabla_{{\bf a}}{G} = 1\oslash {\bf a}$, where ``$\oslash$'' means the element-wise division operation.  Similar expressions can be derived for $\nabla_{{\bf b}}\tilde{F}$ and $\nabla_{{\bf c}}\tilde{F}$. ${\tilde{H}_{\bf a}}= \nabla_{\bf a}^2{F}+\gamma\,\nabla_{\bf a}^2{G}$ is the Hessian of the objective function $\tilde{F}$ with respect to the first factor matrix and $\nabla_{\bf a}^2{G} =-{\rm diag}(1\oslash (\bf a \ast \bf a))$. Similar expressions can be established for ${\tilde{H}_{\bf b}}$ and ${\tilde{H}_{\bf c}}$. Indeed, the neurodynamic system \eqref{neuro_logbar}-\eqref{neuro_logbar3} is the preconditioned gradient flow corresponding to minimization problem \eqref{op_obje}. In the simulation section, we also report the performance of this technique for computing the nonnegative CPD of tensors.

\begin{algorithm}
\caption{CNO-CPD}\label{ALG:TSVDP}
\begin{algorithmic}[1]
{
\STATE{{\bf Input:} A data tensor $\underline{\bf X}$ and a target tensor rank $R$}
\STATE{{\bf Output:} The nonnegative factor matrices ${\bf A},\,{\bf B},\,{\bf C}$ of the CPD}
\STATE{Initializing the state of each RNN randomly ${\bf x}_n^{(0)},\,n=1,2,\ldots,q,$}
\STATE{Set the best local position of particles $p_n^{(0)}=\xx_n^{(0)},\,n=1,2,\ldots,q,$}
\STATE{ Set the best global position of all particles as $p^{(0)}_{best}=\arg\min\{f(\xx_n^{(0)})\},\,n=1,2,\ldots,q$ }
\STATE{Set the CNO and PSO parameters, the diversity threshold $\delta$, error bound $\epsilon$ and the maximum number of iterations $k_{max}$}
\WHILE{$\|f(\xx_n^{(k+1)})-f(\xx_n^{(k)})\|\geq\epsilon$ or $k<k_{max}$}
\STATE{Solve RNNs ($n=1,2,\ldots,q$) formulated as dynamic systems \eqref{CON_1}-\eqref{CON_33} and find their corresponding equilibrium points;}
\STATE{Update the best position of each particle and the best global position of all particles using \eqref{localsearch} and \eqref{globalsearch}}
\STATE{Update the initial conditions of RNNs using Formula \eqref{PSO}-\eqref{PSO2};}
 \IF{${DI<\delta}$}
 \STATE{Perform the wavelet mutation \eqref{Wavelet} and \eqref{Wavelet2}}
 \ENDIF
\STATE{$k=k+1$;}
\ENDWHILE
}
\end{algorithmic}
\end{algorithm}

The convergence behavior of the CNO-CPD is discussed in the next section.

\section{Convergence analysis of the CNO-CPD}\label{Con:CNO}
The constrained optimization problem \eqref{NonOPT} is nonconvex, so the classical iterative algorithms do not guarantee the convergence to a global minimum. Here, we discuss the potential of Algorithm \ref{ALG:TSVDP} (CNO-CPD) for achieving this goal. The Karush–Kuhn–Tucker (KKT) condition as the first optimal conditions necessarily should be satisfied. It has been proved in \cite{yan2014collective} that the equilibrium points of a one-layer RNN and KKT points of a global constrained optimization problem are equivalent as stated in the following lemma.

\begin{lemma} \cite{yan2014collective}
The equilibrium points of a one-layer RNN and KKT points of a global constrained optimization problem have one-to-one correspondence.
\end{lemma}

To prove the global convergence of the CNO-CPD, we can use the existing theorems from the stochastic optimization framework as the PSO algorithm belongs to this category \cite{uryasev2013stochastic} and basically exploit it in our formulation. To do so, we assume that $\Omega$ is the feasible set of the objective function of the constrained optimization problem \eqref{OP_1} and make the following key assumptions in our analysis:

\begin{itemize}
    \item {\bf Assumption I}. If $f({\bf x}^{(k+1)\star})\leq f({\bf x}^{(k) \star})$, then $f({\bf x}^{(k+1) \star})\leq f(\widehat{\bf x})$, where ${\bf x}^{(k) \star}$ denotes the global best solution computed at iteration $k$ and $\widehat{\bf x}$ is an intermediate solution during the iteration.
    %for any $\widehat{\bf x}\in\Omega$ (the current solution obtained during each iteration).
    
    \item {\bf Assumption II}. Given an arbitrary Borel subset $B$ of $\Omega$ ($ B\subset \Omega$), with a discrete Lebesgue
measure $\mu(B)>0$, assume 
    \begin{equation}
        \prod_{k=1}^{\infty}(1-\mu_k(B))=0, 
    \end{equation}
where $\mu_k(B)$ is the conditional probability of achieving a subset $B$ with the measure $\mu_k$, which is defined as follows 
\[\mu_k(B)=P({\bf x}^{(k)}\in B|{\bf x}^{(0)},{\bf x}^{(1)},\ldots,{\bf x}^{(k-1)}).\]
\end{itemize}
Clearly, Assumption I guarantees the monotonically nonincreasing property of the solutions obtained by the CNO-CPD while Assumption II means that with a probability of zero the CNO-CPD fails to find a point of $B$ after infinite time. The next lemma shows that under assumptions I and II, any stochastic process can obtain the global optimum almost surely.
\begin{lemma}\label{Lem_1}\cite{solis1981minimization}
Let $f$ and $\Omega\in\mathbb{R}^n$ be a measurable function and a measurable set, respectively. Assume that Assumptions I and II are satisfied and $\Omega^{g}$ is the set of global points of $f$ over $\Omega$, then for the sequence ${\bf x}^{(k)}$ obtained by a stochastic procedure, we have 
\[
\lim_{k \to +\infty} P({\bf x}^{(k)}\in \Omega^{g})=1,
\]
which means with probability one, we will achieve the global minimum given any initial point.
\end{lemma}
Since we use PSO as a stochastic process in the CNO-CPD, similar to \cite{fan2016collective}, we can prove that for the CNO-CPD with wavelet mutation, Assumptions I and II automatically hold. So according to Lemma \ref{Lem_1}, we conclude the global convergence of the CNO-CPD is guaranteed. So we can present the following theorem.

\begin{theorem}\label{thm_glob}
The CNO-CPD {is globally convergent} with probability one. 
\end{theorem}

\begin{proof}
To prove this theorem, we need to show that the CNO-CPD algorithm satisfies both Assumptions I and II. The sequence $\{{\bf x}^{(k)}\}_{k=0}^{\infty}$ generated by the CNO-CPD algorithm has the monotonically non-increasing property because of using the updating rule \eqref{localsearch}. So the Assumption I holds. 

Assume that the support of the measure $\mu$ for a swarm with a population size of $q$ at iteration $k$ is denoted by $H(\mu_k)=\cup_{i=1}^{q}H(\mu_k^i)$, where $H(\mu_k^i)$ is the support of the $i$-th RNN. Due to using wavelet mutation and re-initializing the RNNs from the feasible region $\Omega$, the actual support is $H(\mu_k)=\cup_{i=1}^{q}H(\mu_k^i)\cup\Omega$, which means that the search space $S$ covers the feasible region. So Assumption II is also held. From Lemma \ref{Lem_1}, the proof is completed.
\end{proof}

The results presented in this section shows that the CNO-CPD algorithm using sufficient number of networks and iterations produces a sequence of non-increasing iterates which almost surely converges to a global optimum point. 
\section{CPD with Discrete-time projection neural network (DTPNN)}\label{Sec:DTPNN}
The discretized version of the continuous dynamic system \eqref{RNN_1} based on the Euler method leads to the following discrete-time projection neural network (DTPNN)  
\begin{equation}\label{RNN_2}
\xt_{k+1}=\xt_k+\lambda_k(-\xt_{k}+P_{\Omega}(\xt_k-\nabla f(\xt_k))),
\end{equation}
where $\lambda_k$ is the step size related of the Euler discretization. We can easily derive a similar discrete-time projection neural network for the nonnegative CPD, which is the discrete-time version of the {neurodynamic} model \eqref{CON_1}-\eqref{CON_33} presented as follows:
\begin{eqnarray}\label{CON_2}
%\left\{\begin{aligned}
\ab_{k+1}&=&\ab_k+\lambda_k(-\ab_k+P_{\Omega}(\ab_k-\nabla_{\bf a} F(\ab_k,\bb_k,\cc_k))),\\\label{CON_22}
\bb_{k+1}&=&\bb_k+\lambda_k(-\bb_k+P_{\Omega}(\bb_k-\nabla_{\bf b} F(\ab_k,\bb_k,\cc_k))),\\\label{CON_223}
\cc_{k+1}&=&\cc_k+\lambda_k(-\cc_k+P_{\Omega}(\cc_k-\nabla_{\bf c} F(\ab_k,\bb_k,\cc_k))).
%\end{aligned}\right.
\end{eqnarray}
The iterations in \eqref{CON_2}-\eqref{CON_223} are fully explicit but of course the iterates can be updated in the Gauss–Seidel manner as presented in Algorithm \ref{ALG:TSVD_2} and {is known to be more efficient \cite{burden1997numerical,golub2013matrix}.} Moreover, we can include more implicitness to have better stability for the model. More precisely, the neurodynamic model \eqref{CON_2}-\eqref{CON_223}  is replaced by
\begin{eqnarray}\label{CON_3}
%\left\{\begin{aligned}
\ab_{k+1}&=&\ab_k+\lambda_k(-\ab_{k+1}+P_{\Omega}(\ab_k-\nabla_{\bf a} F(\ab_k,\bb_k,\cc_k))),\\
\bb_{k+1}&=&\bb_k+\lambda_k(-\bb_{k+1}+P_{\Omega}(\bb_k-\nabla_{\bf b} F(\ab_k,\bb_k,\cc_k))),\\
\cc_{k+1}&=&\cc_k+\lambda_k(-\cc_{k+1}+P_{\Omega}(\cc_k-\nabla_{\bf c} F(\ab_k,\bb_k,\cc_k))).
%\end{aligned}\right.
\end{eqnarray}
which can be simplified to 
\begin{eqnarray}\label{CON_4}
\ab_{k+1}&=&\frac{1}{\lambda_k+1}(\ab_k+ P_{\Omega}(\ab_k-\nabla_{\bf a} F(\ab_k,\bb_k,\cc_k))),\\\label{CON_41}
\bb_{k+1}&=&\frac{1}{\lambda_k+1}(\bb_k+ P_{\Omega}(\bb_k-\nabla_{\bf b} F(\ab_k,\bb_k,\cc_k))),\\\label{CON_42}
\cc_{k+1}&=&\frac{1}{\lambda_k+1}(\cc_k + P_{\Omega}(\cc_k-\nabla_{\bf c} F(\ab_k,\bb_k,\cc_k))).
\end{eqnarray}
Moreover, these techniques can be further accelerated by using the Hessian preconditioning as in \eqref{CON_1}-\eqref{CON_33} for the continuous case. For example, the gradient terms $\nabla_{\bf a} F(\ab_k,\bb_k,\cc_k)$ in \eqref{CON_4}-\eqref{CON_42} are substituted by $\nabla_{\bf a} F(\ab_k,\bb_k,\cc_k){\bf P}^{-1}_a$ where 
$${\bf P}_a= (\C^{T}\C)\ast(\B^T\B)+\delta{\bf I}$$ 
$\delta$ is a regularization parameter and ${\bf I}$ is the identity matrix. An alternating method is to use the cubic regularization trick \cite{nesterov2006cubic}. 
Efficiency of all the variants will be demonstrated in the experiment section.  

The interesting question is the connection between the equilibrium points of continuous-time neurodynamic \eqref{CON_1}-\eqref{CON_33} and discrete-time neurodynamic \eqref{CON_2}-\eqref{CON_223}. The next Lemma responds to this question.

\begin{lemma}\label{Lemm}
The equilibrium points of the continuous-time neurodynamic \eqref{CON_1}-\eqref{CON_33} and discrete-time neurodynamic \eqref{CON_2}-\eqref{CON_223} are equivalent.
\end{lemma}

\begin{proof}
Similar to the proof presented in \cite{li2020discrete}, we can see that the necessary conditions for the CTPNN \eqref{CON_1}-\eqref{CON_33} to be stable are
\begin{eqnarray*}\label{CONM}
%\left\{\begin{aligned}
\epsilon_1\frac{d{\bf a}}{dt}&=&-{\bf a}+P_{\Omega}({\bf a}-\nabla_{{\bf a}}F({\bf a},{\bf b},{\bf c}))=0,\\
\epsilon_2\frac{d{\bf b}}{dt}&=&-{\bf b}+P_{\Omega}({\bf b}-\nabla_{{\bf b}}F({\bf a},{\bf b},{\bf c}))=0,\\
\epsilon_3\frac{d{\bf c}}{dt}&=&-{\bf c}+P_{\Omega}({\bf c}-\nabla_{{\bf c}}F({\bf a},{\bf b},{\bf c}))=0,
%\end{aligned}\right.
\end{eqnarray*}
which means $\nabla_{{\bf a}}F=\nabla_{{\bf b}}F=\nabla_{{\bf c}}F=0$.
Also, the necessary conditions for the DTPNN \eqref{CON_2}-\eqref{CON_223} to be stable are
\begin{eqnarray*}\label{CONN}
%\left\{\begin{aligned}
\ab_{k+1}-\ab_k&=&+\lambda_k(-\ab_k+P_{\Omega}(\ab_k-\nabla_{\bf a} F(\ab_k,\bb_k,\cc_k)))\\ &=&0,\hspace{5.5cm}\\
\bb_{k+1}-\bb_k&=&+\lambda_k(-\bb_k+P_{\Omega}(\bb_k-\nabla_{\bf b} F(\ab_k,\bb_k,\cc_k)))\\ &=&0,\hspace{5.5cm}\\
\cc_{k+1}-\cc_k&=&+\lambda_k(-\cc_k+P_{\Omega}(\cc_k-\nabla_{\bf c} F(\ab_k,\bb_k,\cc_k)))\\ &=&0.\hspace{5.5cm}
%\end{aligned}\right.
\end{eqnarray*}
or $\nabla_{{\bf a}}F=\nabla_{{\bf b}}F=\nabla_{{\bf c}}F=0$. So, the CTPNN \ref{CON_1} and the DTPNN \ref{CON_2} have the same gradient. This means that they have equivalent equilibrium points and this completes the proof.  
\end{proof}
From Lemma \ref{Lemm}, it turns out that setting appropriate step-size $\lambda$ and using the discrete-time neurodynamic \eqref{CON_2}-\eqref{CON_223}, we can converge to equilibrium points of the continuous time neurodynamic \eqref{EQ:4}. We will discuss the upper bound of the step-size $\lambda_k$ for the DTPNN \eqref{CON_2}-\eqref{CON_223} to be convergent in Theorem \ref{Thmm}. The discrete-time neural network \eqref{CON_2}-\eqref{CON_223} with a small step-size $\lambda_k$ is not fast and  for large step sizes has unstable behavior (divergent). So, the step-size should be adaptively chosen at each iteration. For instance the Armijo inequality
\begin{equation}\label{Armijo}
f({\bf x}_{k+1})-f({\bf x}_k)<\alpha\lambda_k\nabla f({\bf x}_k)^T({\bf x}_{k+1}-{\bf x}_k),
\end{equation}
can be used where $0<\alpha<1$. We have used the backtracking approach as an efficient line search method, that is, at each iteration, the Armijo inequality \eqref{Armijo} is checked and if it is not satisfied, the step-size is multiplied by a {positive} constant $0<\beta<1$, and again the Armijo inequality is checked for the new obtained step-size. This procedure is summarized in Algorithm \ref{ALG:TSVD_2},
where
\begin{eqnarray}
\nonumber
R({\bf A}^{(k)},{\bf B}^{(k)},{\bf C}^{(k)})=F({\bf A}^{(k+1)},{\bf B}^{(k+1)},{\bf C}^{(k+1)})
-F({\bf A}^{(k)},{\bf B}^{(k)},{\bf C}^{(k)}),
\end{eqnarray}
%and 
% \begin{eqnarray}
% F({\bf A}^{(k)},{\bf B}^{(k)},{\bf C}^{(k)})=\|\underline{\bf X}-\llbracket {\bf A}^{(k)}, {\bf B}^{(k)},{\bf C}^{(k)}\rrbracket\|^2_F.
% \end{eqnarray}
Each discrete-time projection neural network \eqref{CON_2}-\eqref{CON_223} converges to an equilibrium point of \eqref{CON_1}-\eqref{CON_33}, or a local minimum of \eqref{NonOPT}. To approach a global minimum, we can develop collaborative discrete-time projection neural networks by using multiple  discrete-time projection neural networks so that they can collaborate using PSO.

\begin{algorithm}
\caption{DTPNN}
\label{ALG:TSVD_2}
\begin{algorithmic}[1]
\STATE{{\bf Input:} A data tensor $\underline{\bf X}$, a target tensor rank $R$, maximum of iterations $k_{max}$, an error bound $\epsilon$, parameters $\alpha,\,\beta$, ${\lambda}_k^{\bf a},\,{\lambda}_k^{\bf b},\,{\lambda}_k^{\bf c}$ }
\STATE{{\bf Output:} The nonnegative factor matrices ${\bf A},\,{\bf B},\,{\bf C}$ of the CPD}
\STATE{Define $P:=T:=\{ \{1\},\ldots,\{d\}$\}, $k=0$}
\WHILE{$\|R({\bf A}^{(k)},{\bf B}^{(k)},{\bf C}^{(k)})\|\geq\epsilon$ or $k<k_{max}$}

\WHILE{\eqref{Armijo} is not satisfied}
\STATE{
${\lambda}^{\ab}_{k}{\xleftarrow{}}\beta{\lambda}^{\ab}_k$\\
$\ab_{k+1}=\ab_k+{\lambda}^{\ab}_k(-\ab_k+[\ab_k-\nabla_{\bf a} F(\ab_k,\bb_k,\cc_k)]_{+})$
}
\ENDWHILE
\WHILE{\eqref{Armijo} is not satisfied}
\STATE{${\lambda}^{\bb}_{k}{\xleftarrow{}}\beta{\lambda}^{\bb}_k$\\
$\bb_{k+1}=\bb_k+{\lambda}^{\bb}_k(-\bb_k + [\bb_k-\nabla_{\bf b} F(\ab_{k+1},\bb_k,\cc_k)]_{+})$}
\ENDWHILE

\WHILE{\eqref{Armijo} is not satisfied}
\STATE{
${\lambda}^{\cc}_{k}{\xleftarrow{}}\beta{\lambda}^{\cc}_k$\\
$\cc_{k+1}=\cc_k+{\lambda}^{\cc}_k(-\cc_k + [\cc_k-\nabla_{\bf c} F(\ab_{k+1},\bb_{k+1},\cc_k)]_{+})$
}
\ENDWHILE
\STATE{$k=k+1$}
\ENDWHILE
\end{algorithmic}
\end{algorithm}

We now present two lemmas that are used in our stability and convergence analyses. Lemma \ref{Lem2} demonstrates that the neurodynamcis model \eqref{CON_2}-\eqref{CON_223} can be represented in an alternative form. This result along with Lemma \ref{Lem1}, help to derive an upper bound on the step-size $\lambda_k$ for the DTPNN to converge to an equilibrium point of neurodynamic system \eqref{CON_2}-\eqref{CON_223}. Note that Lemma \ref{Lem2} is for the general case of the activation function \eqref{acfunc}.

\begin{lemma}\label{Lem1} \cite{kinderlehrer2000introduction}
Given $\Omega \subseteq \mathbb{R}^n$ convex; for any $\xx \in \mathbb{R}^n$ and any $\yy \in \Omega$, we have
\[
(P_{\Omega}(\xx)-\xx)^T(\yy-P_{\Omega}(\xx))\geq 0.
\]
\end{lemma}

\begin{lemma}\label{Lem2} Assume $\qq_{k,{\bf a}}=\ab_k-\nabla F_{\bf a}(\ab_k,\bb_k,\cc_k)$, $\qq_{k,{\bf b}}=\bb_k-\nabla F_{\bf b}(\ab_k,\bb_k,\cc_k)$, $\qq_{k,{\bf c}}=\cc_k-\nabla_{\bf c} F(\ab_k,\bb_k,\cc_k)$, then there exist two functions $\alpha(.)$ and $\beta(.)$ for which the following discrete-time projection neural network
\begin{eqnarray}\label{CO}
%\left\{\begin{aligned}
\ab_{k+1}^{i}&=&\ab_k^{i}+\lambda_k(-\ab^{i}_k+P_{\Omega}(\qq^{i}_{k,{\bf a}})),\\
\bb^{i}_{k+1}&=&\bb^{i}_k+\lambda_k(-\bb^{i}_k+P_{\Omega}(\qq^{i}_{k,{\bf b}})),\\
\cc^{i}_{k+1}&=&\cc^{i}_k+\lambda_k(-\cc^{i}_k+P_{\Omega}(\qq^{i}_{k,{\bf c}})).
%\end{aligned}\right.
\end{eqnarray} 
is equivalent to 
\begin{eqnarray}\label{CO_2}
%\left\{\begin{aligned}
\ab^{i}_{k+1}&=&\ab^{i}_k-\gamma^i_k(\nabla_{\bf a} f)^{i},\\
\bb^{i}_{k+1}&=&\bb^{i}_k-\gamma^i_k(\nabla_{\bf b} f)^{i},\\
\cc^{i}_{k+1}&=&\cc^{i}_k-\gamma^i_k(\nabla_{\bf c} f)^{i},
%\end{aligned}\right.
\end{eqnarray} 
% for $i=1,2,\ldots,I_1,\,j=1,2,\ldots,I_2,\,\,t=1,2,\ldots,I_3,$
where 
\begin{eqnarray}\label{CO_3}
\gamma^i_k=\left\{\begin{aligned}
\lambda^i_k,\quad\qq_{k,\tau}^{i}\in[l_i,u_i],\,\,\,\tau={\bf a},\,{\bf b},\,{\bf c}\,\\
\lambda^i_k\alpha_i(\tau_k),\,\,\,\qq_{k,\tau}^{i}>u_i,\,\,\,\tau={\bf a},\,{\bf b},\,{\bf c}\\
\lambda^i_k\beta_i(\tau_k),\,\,\,\qq_{k,\tau}^{i}<l_i,\,\,\,\tau={\bf a},\,{\bf b},\,{\bf c}.\end{aligned}\right.
\end{eqnarray}

\end{lemma}

\begin{proof}
See the Appendix.
\end{proof}

\section{Stability and convergence analysis of DTPNN}\label{Sec:stab}
In this section, we discuss the stability and convergence analyses of the proposed approach. We first start with the stability analysis. Here, the main question which needs to be answered is for what $\lambda_k$, the discrete-time neural network \eqref{CON_2}-\eqref{CON_223} is stable and converges to an equilibrium point of \eqref{CON_1}-\eqref{CON_33}. The next Lemma provides such information for a one-layer projection neural network \eqref{RNN_2}.

\begin{lemma}\label{Thmm} \cite{che2018nonnegative}
The one-layer discrete-time projection neural network \eqref{RNN_2} is stable in the Lyapunov sense and converges to one of its equilibrium points for any initial states ${\bf x}_{0}$ provided
\[
\max(0,1-\sqrt{c})\leq\lambda_k\leq(1+\sqrt{c}),
\]
where $c=(1-2\|\alpha_k^2\nabla f({\bf x}_{k})\|_2^2)/(\|[{\bf x}_{k}-\nabla f({\bf x}_{k})]_{+}-{\bf x}_{k}\|_2^2)
$ and $\alpha_k=(\sqrt{\gamma^i_k},\ldots,\sqrt{\gamma^i_k})^T$, $\gamma^i_k$ defined as in \eqref{CO_3}.
\end{lemma}
%$c=(\|\begin{bmatrix}[]\end{bmatrix}\|_2^2+1)(\|\begin{bmatrix}[]\end{bmatrix}\|_2^2)^{-1}
%$
Similarly, for the multi-layer discrete-time neural network \eqref{CON_2}-\eqref{CON_223}, we have the following theorem.
\begin{theorem}\label{Thm_pro}
The discrete-time projection neural network \eqref{CON_2}-\eqref{CON_223}, is stable in the Lyapunov sense and converges to one of its equilibrium points of any initial states $({\bf a}^{(0)},\,{\bf b}^{(0)},\,{\bf c}^{(0)})$ if 
\begin{equation}\label{interval}
\max(0,1-\sqrt{c})\leq\lambda_k\leq(1+\sqrt{c}),
\end{equation}
where 
\begin{equation*}
\small
c=\left(1-2\norm{\begin{bmatrix}\gamma^T_k\nabla_{\bf a} f\\\gamma^T_k\nabla_{\bf b} f\\\gamma^T_k\nabla_{\bf c} f\end{bmatrix}}_2^2\right)\left(\norm{\begin{bmatrix}{\bf a}_k - [{\bf a}_k-\nabla_{\bf a}f]_{+}\\{\bf b}_k-[{\bf b}_k-\nabla_{\bf b}f]_{+}\\{\bf c}_k-[{\bf c}_k-\nabla_{\bf c}f]_{+}\end{bmatrix}}_2^2\right)^{-1}.
\end{equation*}

\end{theorem}

\begin{proof}
See the Appendix.
\end{proof}

Theorem \ref{Thm_pro} is for the DTPNN with a specific step-size mentioned in the interval. Since, in algorithm \ref{ALG:TSVD_2}, the flexible step-size based on the backtracking procedure is utilized, here, we also analyze the stability and convergence of this algorithm.

\begin{theorem}\label{Thmconve}
Let the objective function $f$ is bounded below, the DTPNN algorithm generates triples $({\bf a}_k,{\bf b}_k,{\bf c}_k)$ for which the sequence $\{f({\bf a}_k,{\bf b}_k,{\bf c}_k)\}_{k=1}^{\infty}$ is monotonically non-increasing and reaches an equilibrium point of it. 
\end{theorem}

\begin{proof}
The monotonically non-increasing property of the sequence $\{f({\bf a}_k,{\bf b}_k,{\bf c}_k)\}_{k=1}^{\infty}$ is clear because of \eqref{Armijo} \cite{che2018nonnegative}. To be more precise, due to performing the backtracking method for each factor matrix in Algorithm \ref{ALG:TSVD_2}, we have
\begin{eqnarray}
\nonumber
f({\bf a}_{k+1},{\bf b}_{k+1},{\bf c}_{k+1})\leq f({\bf a}_{k},{\bf b}_{k+1},{\bf c}_{k+1})\leq f({\bf a}_{k},{\bf b}_{k},{\bf c}_{k+1})\leq  f({\bf a}_{k},{\bf b}_{k},{\bf c}_{k}).
\end{eqnarray}
The sequence $\{f({\bf a}_k,{\bf b}_k,{\bf c}_k)\}_{k=1}^{\infty}$ is monotonically non-increasing and bounded, so it is convergent. In view of \eqref{Armijo}, and considering $\nabla f({\bf x}_k)^T({\bf x}_{k+1}-{\bf x}_k)\geq 0$, we conclude that in the convergent case, we should have $\nabla f({\bf x}_k)=0$ and this completes the proof.
\end{proof}

It is proved in \cite{che2018nonnegative} that a stationary point which {in} our case is an equilibrium point of a DTPNN is a partial optimum. So, from this fact and also Theorem \ref{Thmconve}, we conclude that the DTPNN algorithm is convergent to a partial optimum.
Note that one can use a collection of DTPNN models so that they can collaborate with each other through the PSO method and we call this approach CNO-DTPNN. Combining results presented in Theorems \ref{thm_glob} and \ref{Thmconve}, we deduce that the CNO-DTPNN convergences to a global minimum with probability one. In the experimental results we see the CNO-DTPNN can provide relatively better results than the DTPNN method.
\section{Experimental Results}\label{Sec:Sim}
%For some cases we have used supercomputer/cluster with 120Gb RAM. 
In this section, we present our experimental results. We have used \MATLAB to implement the algorithms on a PC computer with 2.60 GHz Intel(R) Core(TM) i7-5600U processor and 8GB memory. The parameters of the PSO are set as $\alpha=0.5,\,\beta_1=\beta_2=0.01$ and $\gamma_1={\rm rand}(1),\,\gamma_2={\rm rand}(1)$ where the MATLAB command ${\rm rand}(1)$ generates a random number in $(0,1)$ uniformly. The code is available from \url{https://github.com/vleplat/Neurodynamics-for-TD}.

The relative error is defined as follows
\[
{\rm Relative\,\,error}=\frac{\|\underline{\bf X}-\llbracket {\bf A}^{(1)}, {\bf A}^{(2)},{\bf A}^{(3)}\rrbracket\|_F}{\|\underline{\bf X}\|_F}.
\]

We have compared the proposed CNO-based algorithm with HALS \cite{cichocki2009fast,cichocki2007hierarchical,cichocki2009nonnegative}, MUR \cite{cichocki2009nonnegative}, CCG \cite{farias2014data}, CGP \cite{royer2011computing}, BFGSP \cite{royer2011computing}, GradP \cite{royer2011computing} algorithms. The implementation of many of nonnegative tensor decomposition algorithms can be found in \cite{Comontoolbox}.

%\begin{exa}
EXAMPLE 1. ({\bf Decomposing difficult tensors}) In this experiment, we consider the challenging data tensors where the tensor rank exceeds the tensor dimension. To this end, we first generated a third-order tensor of size $9\times 9\times 9$ and the tensor rank $R=10$. Here, we examined the efficiency of a single continuous neurodynamic model \eqref{CON_1}-\eqref{CON_33} compared to the baseline methods for decomposing the given tensor. The relative errors of the considered algorithms are compared in Figure \ref{fig_exa_1}. The ANLS algorithm failed for many tests, but we reported the outcome of this method for a successful run. The results demonstrate that the ODE, which refers to the single continuous neurodynamic technique used in our study, exhibits superior performance compared to other techniques.
%The execution times of the algorithms are also shown in Figure \ref{fig_exa_1}. The continuous neurodynamic model requires more computing times while other baselines relatively have similar execution times. 
We also employed larger tensor ranks $R=11,12,\ldots,16$ to check the stability of the single continuous neurodynamic model for such more difficult cases. The results of these experiments are shown in Figure \ref{fig_exa_2} (left). As can be seen, decomposing data tensors with larger tensor ranks are more challenging but the single continuous neurodynamic model still provides satisfying results. We have compared the continuous neurodynamic with classical ODE and log barrier formulations in Figure \ref{fig_exa_2} (right). The numerical experiments show that the classical ODE formulation yields superior results compared to the log barrier formulation.

\begin{figure}
\begin{center}
\includegraphics[width=0.5\columnwidth]{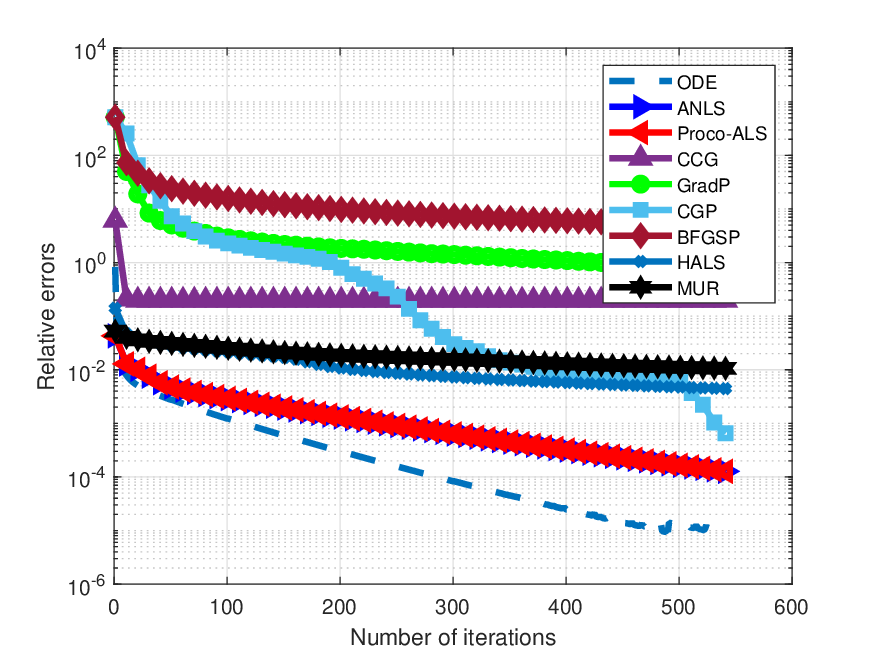}
\caption{\small{Comparing the algorithms for decomposing a nonnegative tensor of size $9\times 9\times 9$ and the tensor rank $R=10$.}}\label{fig_exa_1}
\end{center}
\end{figure}

Moreover, to further check the applicability of the single continuous neurodynamic for tensors of larger sizes, we consider a tensor of size $70\times 70\times 70$ and the tensor rank $R=75$. The relative errors of this simulation are reported in Figure \ref{fig_exa_3} (left). The BFGSP and Proco-ALS algorithms failed for this case. The results clearly show that the single continuous neurodynamic model is more efficient than the baseline methods and is also applicable for tensors of medium size. 

Finally, in order to investigate the efficiency of the discretized algorithm \ref{ALG:TSVD_2}, we decomposed a nonnegative tensor of size $9\times 9\times 9$ and the tensor rank $R=10$, using the different ideas such as the semi-implicit, full-explicit and cubic regularization methods. The results are shown in Figure \ref{fig_exa_3} (right). The results show that the semi-implicit method can provide better results than other techniques.

\begin{figure}
\begin{center}
\includegraphics[width=0.45\columnwidth]{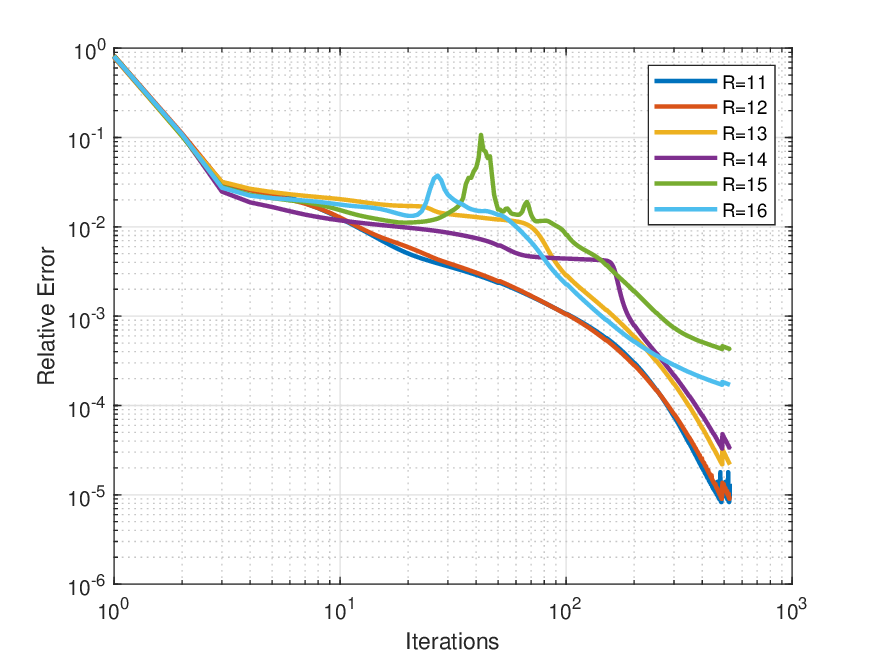}
\includegraphics[width=0.45\columnwidth]{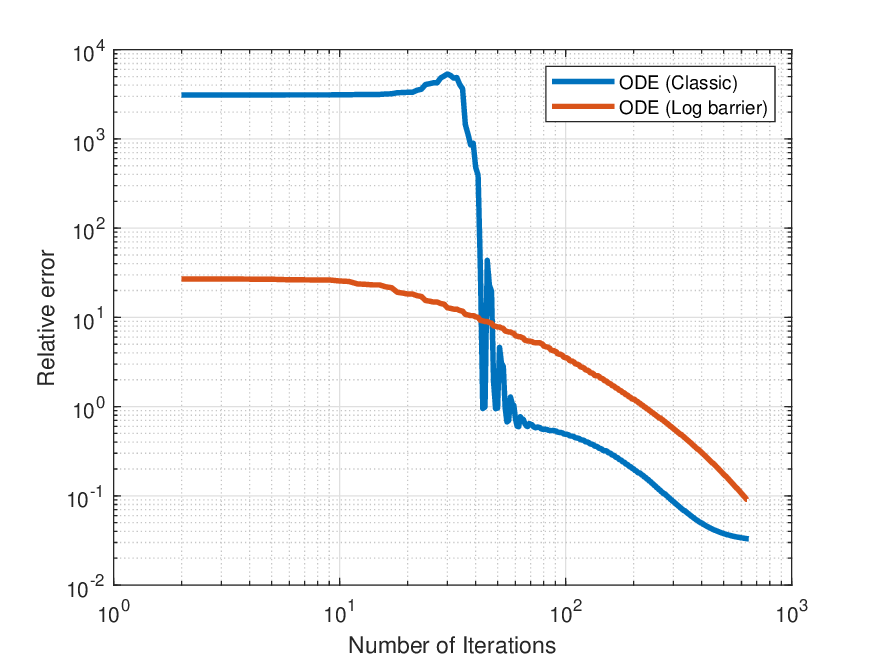}
\caption{\small{(left) The numerical results of the continuous neurodynamic for decomposing a nonnegative tensor of size $9\times 9\times 9$ and the tensor ranks $R=11,12,\ldots,16$. (right) Comparing the continuous neurodynamic and its log barrier formulation for decomposing a nonnegative tensor of size $9\times 9\times 9$ and the tensor rank $R=20$.}}\label{fig_exa_2}
\end{center}
\end{figure}

\begin{figure}
\begin{center}
\subfigure[]{\includegraphics[width=0.45\columnwidth]{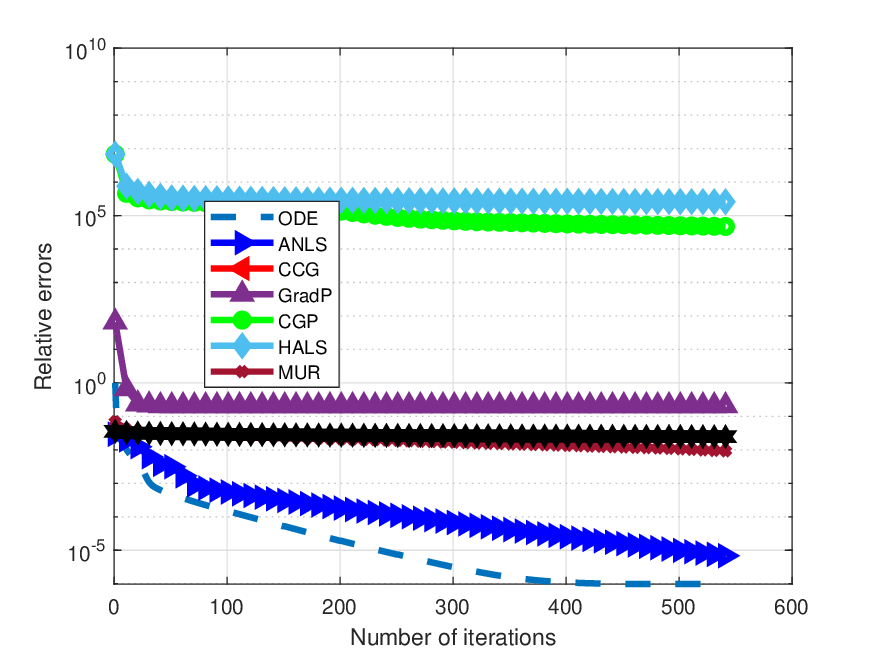}}
\subfigure[]{\includegraphics[width=0.45\columnwidth]{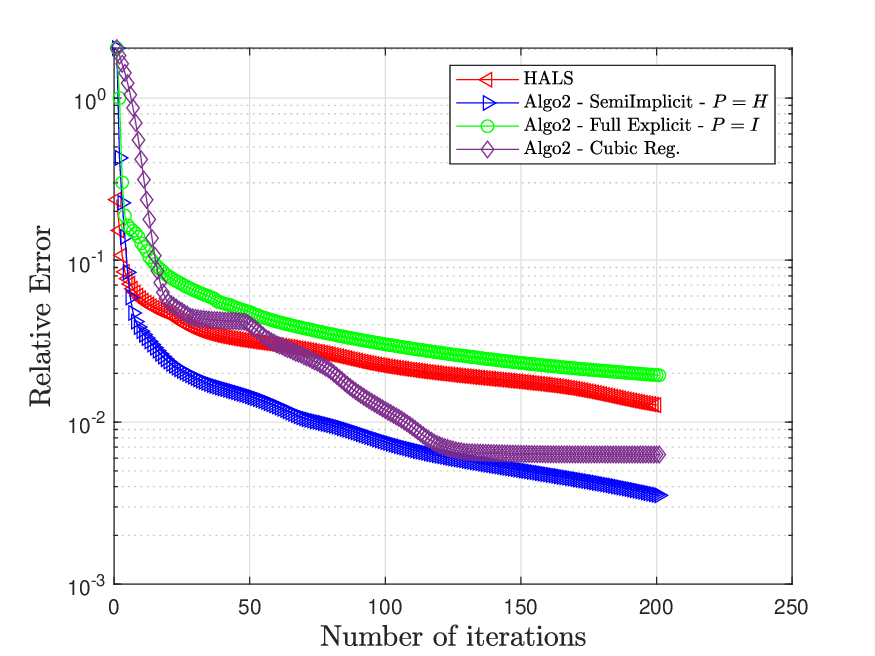}}
\caption{\small{Comparing the algorithms for decomposing a nonnegative tensor of size $70\times 70\times 70$ and the tensor rank $R=75$. The BFGSP and Proco-ALS algorithms failed. (right). {The numerical results of the discrete neurodynamic for decomposing a nonnegative tensor of size $9\times 9\times 9$ and the tensor rank $R=10$.}}}\label{fig_exa_3}
\end{center}
\end{figure}

%\end{exa}

%\begin{exa}
EXAMPLE 2. ({\bf Decomposing tensors with high colinear factor matrices}) In this example, we apply our neurodynamic model on random data tensors with high colinear factor matrices. In particular, we demonstrate the impact of the PSO in getting better results. We first create three matrices ${\bf A}^{(i)}\in\mathbb{R}^{20\times 10},\,i=1,2,3$ whose elements independent and identically distributed uniform random variables and build a 3rd order tensor $\underline{\bf X}\in\mathbb{R}^{20\times 20\times 20}$. We applied the proposed collaborative neurodynamic Algorithm \ref{ALG:TSVDP} to compute a nonnegative CPD with tensor rank $R=10$. We mainly consider the following two scenarios:

\begin{itemize}
      \item {\bf Case study I.} Random data with one high colinear factor matrix.\\
      
        \item {\bf Case study II.} Random data with two  high colinear factor matrices.
\end{itemize}

{
We remind that the collinearity in general (if the column matrices are not normalized to length one) is defined as
\[
\mu = \frac{{{\mathbf a}_m^{\left( r \right)\,T}{\mathbf a}_n^{\left( r \right)}}}{{{{\left\| {{\mathbf a}_m^{\left( r \right)}} \right\|}_2}{{\left\| {{\mathbf a}_n^{\left( r \right)}} \right\|}_2}}},\,\,\,m \ne n
\]
where ${\mathbf a}_m^{\left( r \right)}$ and ${\mathbf a}_n^{\left( r \right)}$ are two columns of the same factor matrix \cite{phan2013low}. 
}
In this example, we have used our CNO-CPD Algorithm with population size $q=5$ in our computations. For Case I, we assume that the third factor matrix $({\bf A}^{(3)})$ is highly colinear $0.96\leq\mu\leq 0.99$ while two other factor matrices (${\bf A}^{(2)},\,{\bf A}^{(3)}$) have collinearity $0.4\leq\mu\leq 0.6$. Here we applied the baseline algorithms and the proposed neurodynamic model to the mentioned data tensor. The relative errors are reported in Figure \ref{Fig_11} (left). 
%Also the transient behaviors of the proposed three-timescale neurodynamic optimization model are displayed in Figure \ref{Fig_1} (upper). 
As can be seen the neurodynamic model has better performance and in a few iterations provide better results than the baseline algorithms. Figure \ref{Fig_11} (left) shows that we have achieved stable solutions. 

For Case II, we assume that the second and third factor matrices (${\bf A}^{(2)},\,{\bf A}^{(3)}$) have collinearity $0.96\leq\mu\leq 0.99$ and the first factor matrix (${\bf A}^{(1)}$) has colinearity $0.4\leq\mu\leq 0.6$. Here again, the baseline algorithms and the proposed neurodynamic model are applied to this new data tensor. The relative errors of them are reported in Figure \ref{Fig_11} (right). 
%Also, the transient behaviors of the proposed three-timescale neurodynamic optimization model (with population size $q=5$) are displayed in Figure \ref{Fig_1} (bottom). Here again, better results for the neurodynamic model are achieved compared to the baseline algorithms. 
%
We performed Monte Carlo experiments using different numbers of population sizes $q=5,\,10,\,15,\,20,\,25,\,30$. These results are displayed in Figure \ref{Figmonte} (left) for the data tensor Case I and Figure \ref{Figmonte} (right) for the data tensor Case II. In view of Figure \ref{Figmonte}, we see that for larger population sizes we achieve a lower relative error.

{To ensure fairness in the comparison and to highlight the significance of using a stable ODE solver in the formulation, we also equipped the best baseline algorithm which was the HALS algorithm with PSO re-initialization. The new HALS-PSO algorithm still had a very slow convergence and almost similar to the original HALS algorithm. While the proposed algorithm generally requires more time due to re-initialization and ODE solver, however for such difficult scenarios where too many iterations are required for the convergence of the baseline algorithms, the CNO-CPD requires less time for the convergence.
The per iteration complexity of the proposed collaborative algorithm is higher than the baselines due to re-initialization process and ODE solver. However, as we see for tensors with high co-linearity, most of the baselines have very slow convergence and need too many iterations for the convergence. In our simulation, it took the CNO-CPD 20 (sec) to get an approximation with $10^{-4}$ accuracy while in 100 (sec) the best baseline (HALS) or HALS with PSO re-initialisation just achieved an approximation with  accuracy $10^{-1}$. These results convinced us that the proposed algorithm can also be beneficial in terms of running time compared to the baselines. Indeed, the proposed algorithm provides more stable solutions in a reasonable amount of computing time. 
}

\begin{figure*}
\begin{center}
\includegraphics[width=0.45\columnwidth]{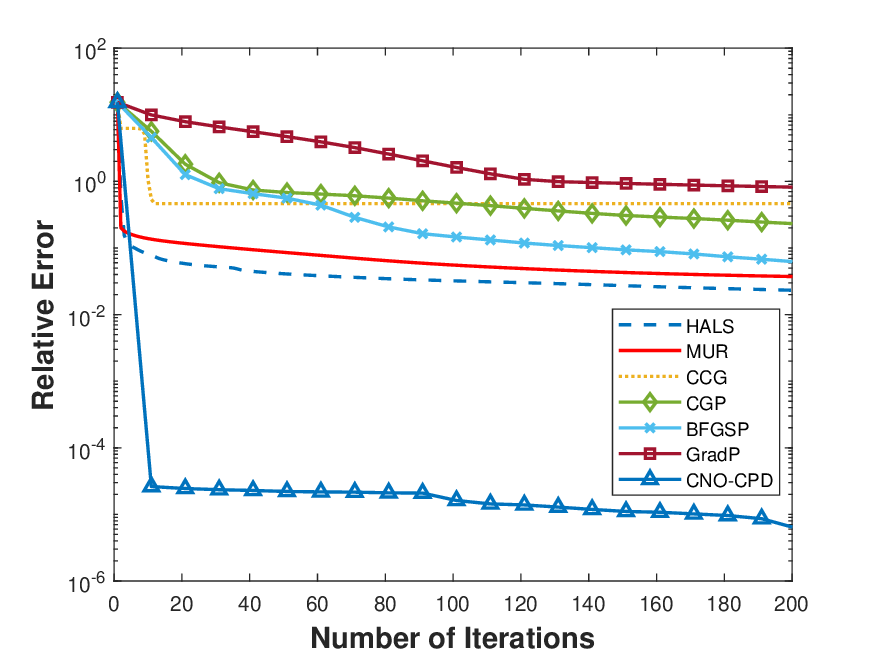}
\includegraphics[width=0.45\columnwidth]{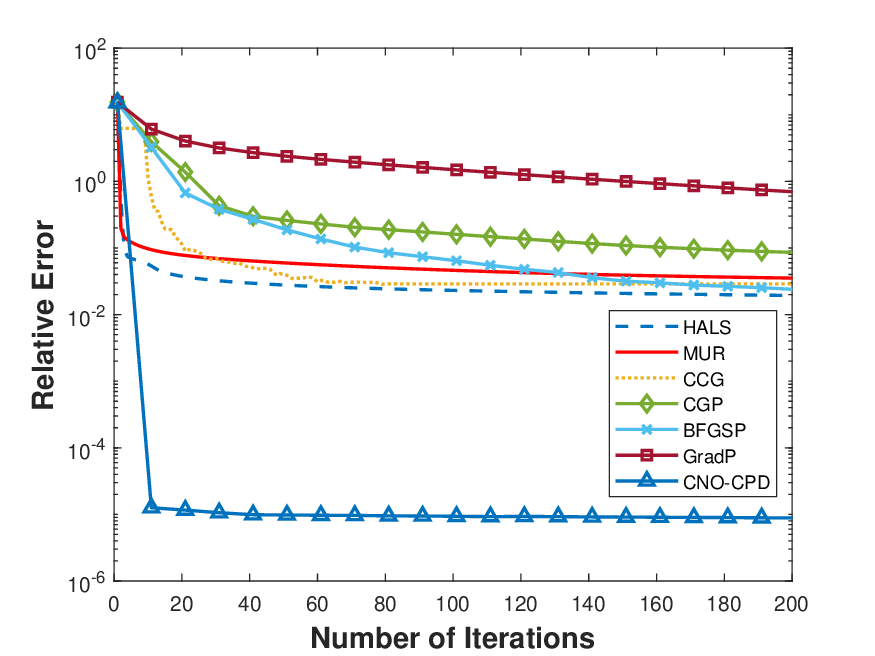}
\caption{\small{(Left) The relative error comparison of the proposed algorithm and baseline algorithms for the random data tensor Case study  I ($q=30$). (Right) The relative error comparison of the proposed algorithm and baseline algorithms for the random data tensor Case study II ($q=30$).}}\label{Fig_11}
\end{center}
\end{figure*}

\begin{figure*}\label{Figmonte}
\begin{center}
\includegraphics[width=0.45\columnwidth]{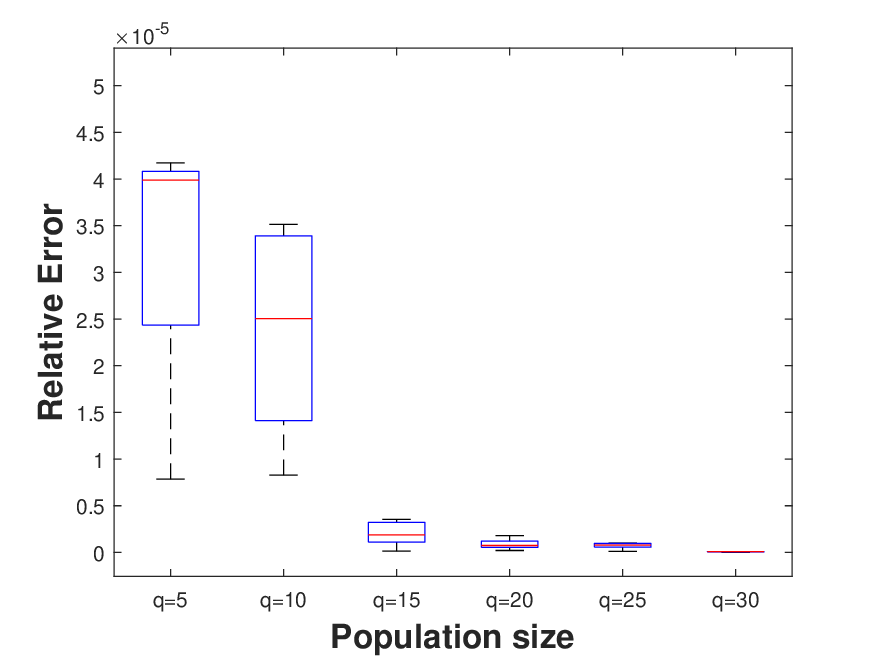}
\includegraphics[width=0.45\columnwidth]{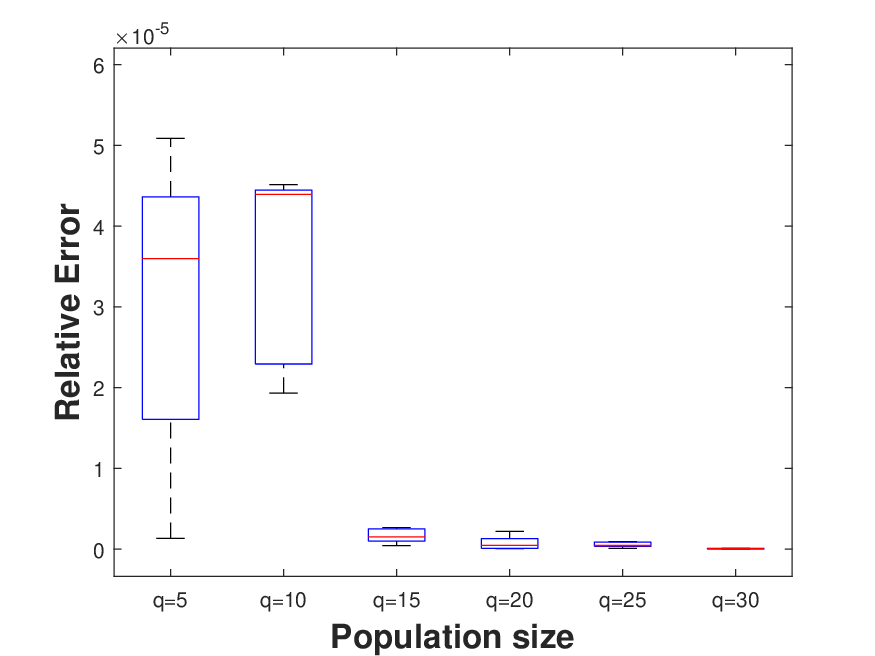}
\caption{\small{Monte Carlo experimental results for the CNO-based algorithm for different population sizes. (Left) The relative error comparison of the proposed algorithm for random data tensor Case study  I. (Right) The relative error comparison of the proposed algorithm for random data tensor Case II.}}\label{Fig_box}
\end{center}
\end{figure*}

%\begin{exa} 
EXAMPLE 3. ({\bf Real-world data sets}) 

In this experiment, we use several real-world data tensors to show the efficiency of the proposed CNO-CPD algorithm. We main consider ``{\it COIL20}''\footnote{\url{https://www.cs.columbia.edu/CAVE/software/softlib/coil-20.php}}, ``{\it YALE}''\footnote{\url{http://cvc.cs.yale.edu/cvc/projects/yalefaces/yalefaces.html}}, and ``{\it ORL}''\footnote{\url{https://cam-orl.co.uk/facedatabase.html}}
%``{\it MNIST}''\footnote{\url{http://yann.lecun.com/exdb/mnist/}} and ``{\it CIFAR10}'' \footnote{\url{https://www.cs.toronto.edu/~kriz/cifar.html}}
datasets in our simulations. 
A detailed information about the mentioned datasets are described below:

1- {\bf COIL20 dataset.} It is a collection of gray-scale images of 20 different objects under 72 different rotations. The size of each image is $32\times 32$ and naturally builds a fourth-order tensor of size $32\times 32\times 20\times 72$. We reshape the data tensor to a 3rd order tensor of size $32\times 32 \times 1440$.\\

2- {\bf YALE dataset.} It is a collection of gray scale images of 15 persons under 11 different expressions such as sad, sleepy, surprised, and wink. Each image is of size $32\times 32$ and the data can be represented as a fourth-order tensor of size $32\times 32\times 15\times 11$. We reshape it to a third-order tensor of size $32\times 32\times 165.$ \\

3- {\bf ORL dataset.} It contains gray-scale images of 10 persons for 40 distinct subjects such as facial expressions (open/closed eyes, smiling/not smiling) and facial details (glasses/no glasses). The size of each is $32\times 32$ and whole data is represented as a fourth-order tensor of size $32\times 32\times 10\times 40$. We reshape the data to a third-order tensor of size $32\times 32\times 400$.\\

We used tensor rank $R=10$ for all data tensors to represent them in the CPD format. {In this experiment, the data tensors are large. So, computing the stable solutions of the continuous neurodynamic is prohibitive because we usually need a large time-step for finding the stable solutions.} To resolve this problem, we have used discrete neurodynamic with a population size of 5. The relative errors of the algorithms for the COIL20, the ORL and the YALE datasets are reported in Figures \ref{COIL20}, respectively. 
From Figure \ref{COIL20}, we see that the CNO-based algorithm can achieve lower objective values. This clearly demonstrates the superiority of the CNO-based algorithms over the baseline methods. 

{It is known that
the factor matrices of the nonnegative CPD can be used for clustering multi-dimensional
datasets \cite{zhou2014decomposition}. This can be considered as an interpretation of the data via nonnegative CPD.
More precisely, the factor matrices obtained via nonnegative CPD extract important
features of the underlying datasets after which the clustering methods can be applied
to them to effectively partition the datasets. Here, we employed the t-distributed stochastic neighbor embedding method \cite{van2008visualizing}. Applying the t-SNE directly to the matricized version of the data is quite expensive and this is why we apply it to the last factor matrix as a matrix extracting important features of the dataset. Indeed, the t-SNE  applying on the last factor matrix returns a matrix of two-dimensional embeddings of the high-dimensional rows of the factor matrix. The clustering results are visualized 
in Figure \ref{Clustering}.}

\begin{figure*}
\begin{center}
\includegraphics[width=0.45\columnwidth]{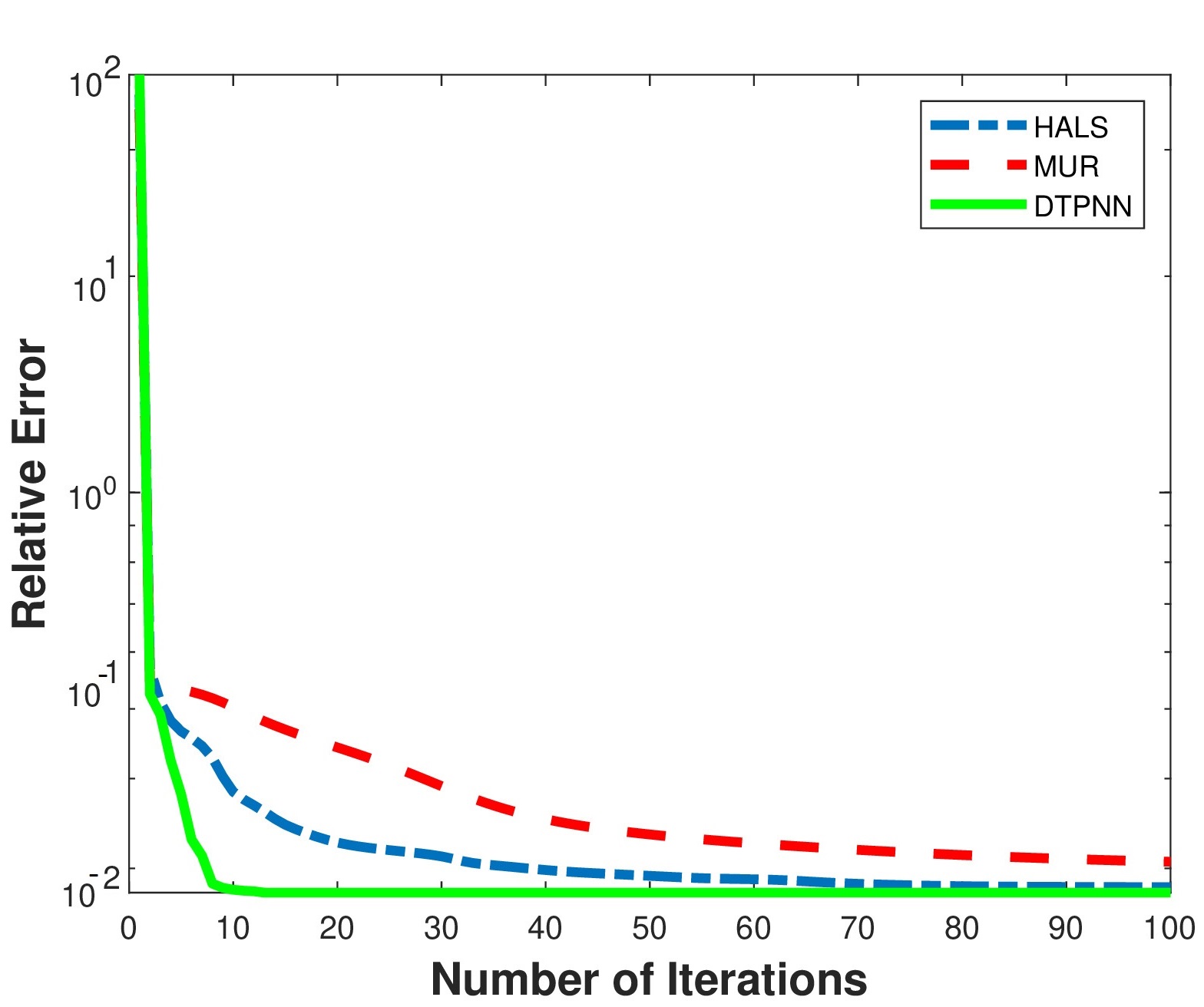}
\includegraphics[width=0.45\columnwidth]{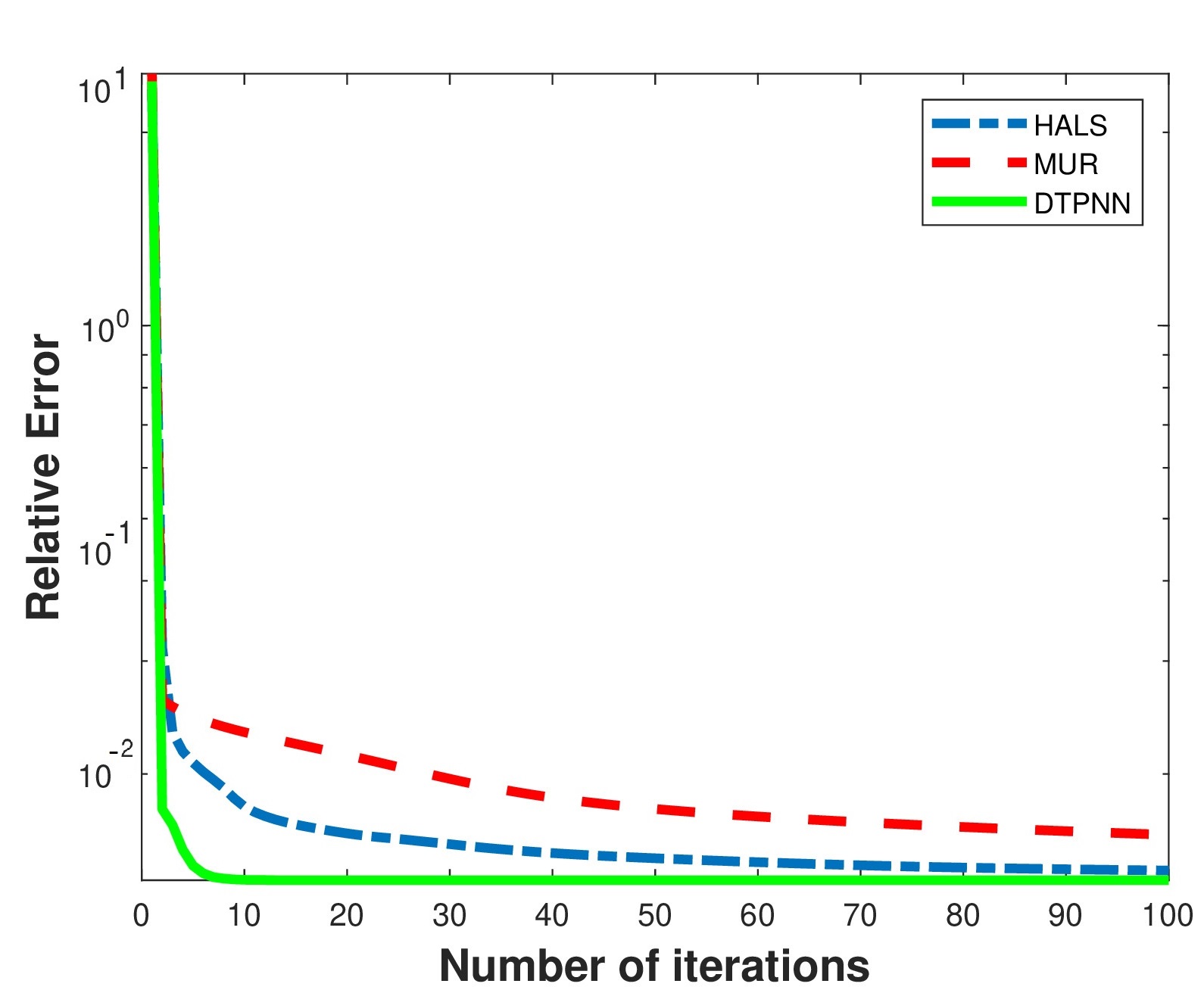}\\
\includegraphics[width=0.45\columnwidth]{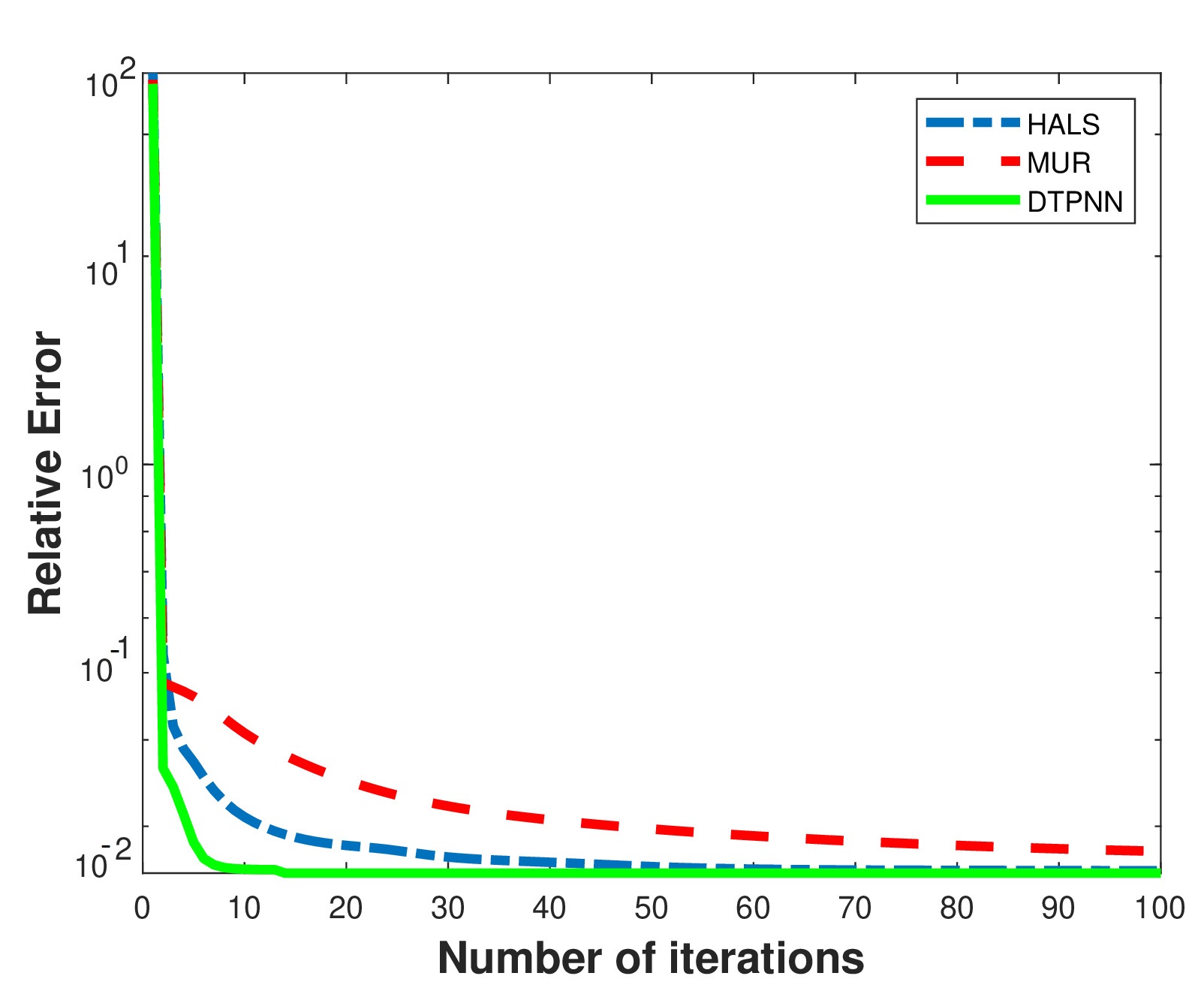}
\caption{\small{(Top Left) The relative error comparison of the proposed algorithm and baseline algorithms for the COIL20 dataset. (Top Right) The relative error comparison of the proposed algorithm and baseline algorithms for the ORL dataset. (Bottom) The relative error comparison of the proposed algorithm and baseline algorithms for the YALE dataset.}}\label{COIL20}
\end{center}
\end{figure*}

\begin{figure}
\begin{center}
\includegraphics[width=.45\columnwidth]{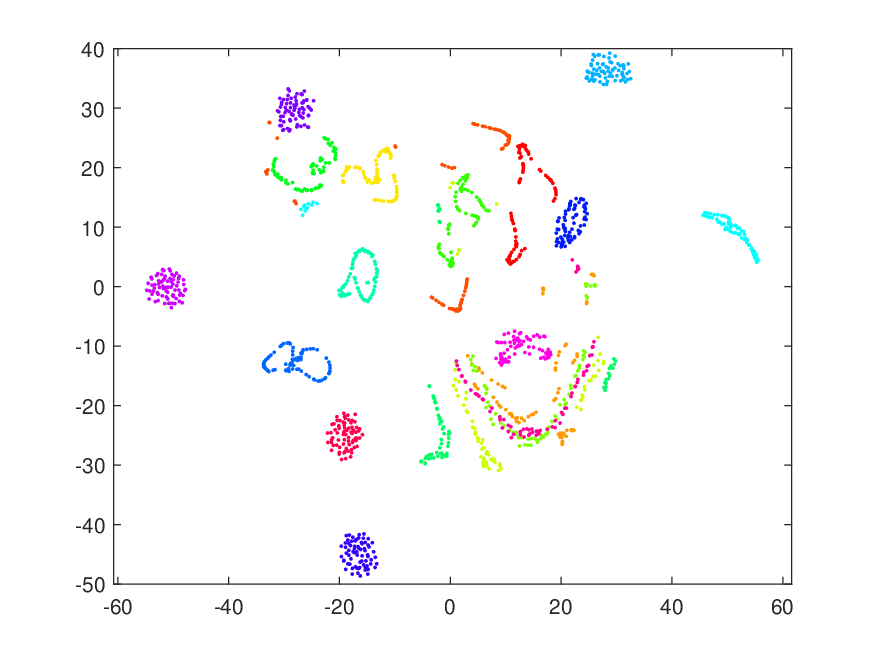}
\includegraphics[width=.45\columnwidth]{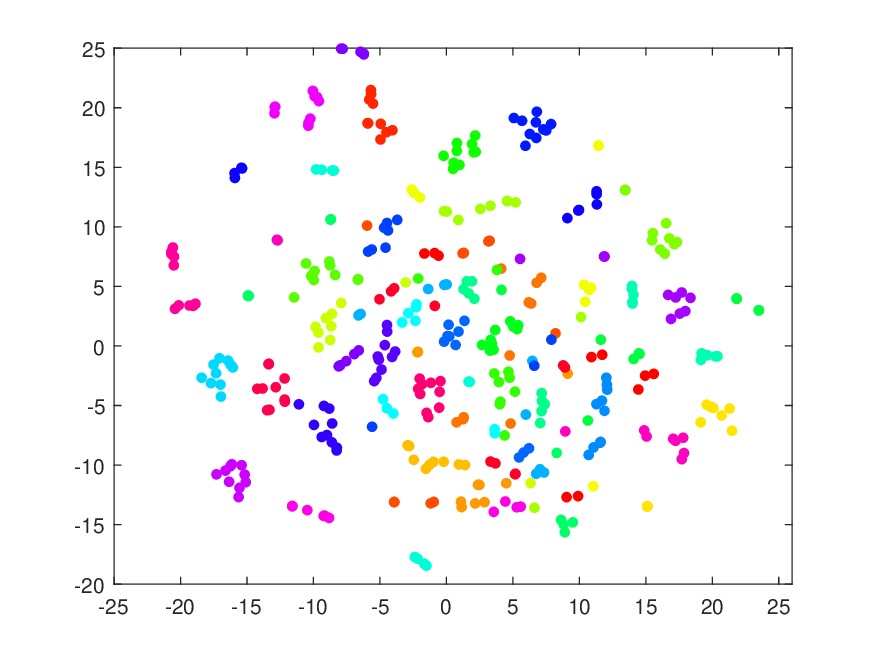}
 \includegraphics[width=.45\columnwidth]{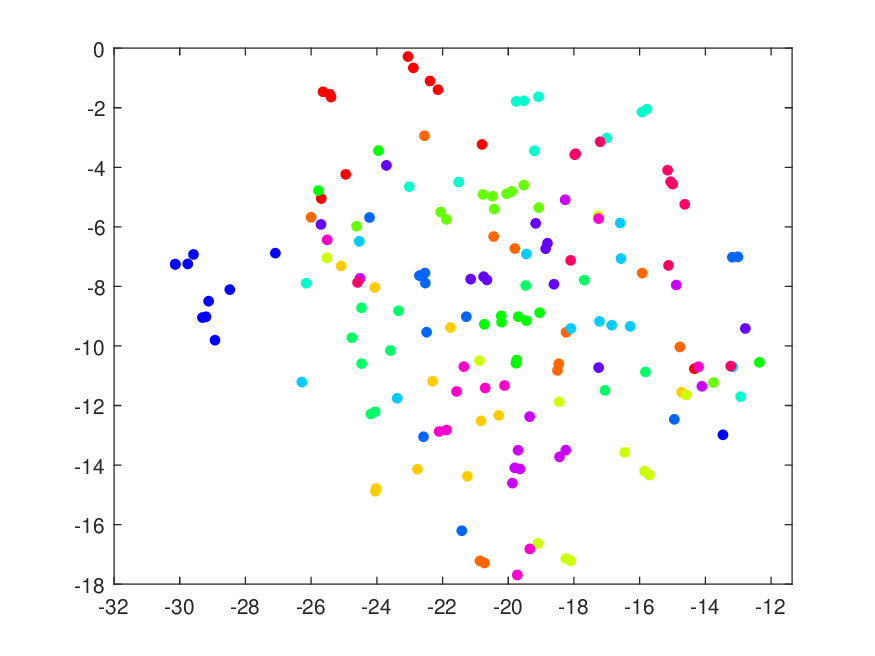}
\caption{\small{Clustering results obtained by the proposed algorithm for (Top left) COIL20 dataset, (Top right) ORL dataset and (bottom) YALE dataset.}}\label{Clustering}
\end{center}
\end{figure}

%\end{exa}

EXAMPLE 4. ({\bf Hyperspectral Images data sets}) 
We test the proposed CNO-CPD algorithms for hyperspectral unmixing~\cite{fevotte2014nonlinear}. In this case, the continuous neurodynamic with classical ODE (PSO deactivated) is the workhorse approach, and we compare it to MUR and HALS methods: Figure~\ref{fig:HSI} 
provide the median evolution of the relative errors values for:
\begin{itemize}
    \item the Cuprite data set\footnote{\label{note1}Available for download at \url{http://lesun.weebly.com/hyperspectral-data-set.html}.}: third-order tensor of size $120 \times 120 \times 180$ containing mostly 12 types of minerals ($R=12$).
    \item the Urban data set\footref{note1}: third-order tensor of size $307 \times 307 \times 162$ containing mostly 6 materials ($R=6$), namely "Road", "Soil", "Tree", "Grass", "Roof top 1" and "Roof top 2". We extract a subtensor of size $120 \times 120 \times 162$.
    \item the Jasper Ridge data set\footref{note1}: third-order tensor of size $100 \times 100 \times 198$ containing mostly 4 materials ($R=4$), namely "Road", "Soil", "Water" and "Tree".
\end{itemize}
One can observe that the CNO-based algorithm achieves lower objective values and demonstrates faster convergence per iteration compared to its competitors.

\begin{figure}
\begin{center}
\includegraphics[width=0.45\columnwidth]{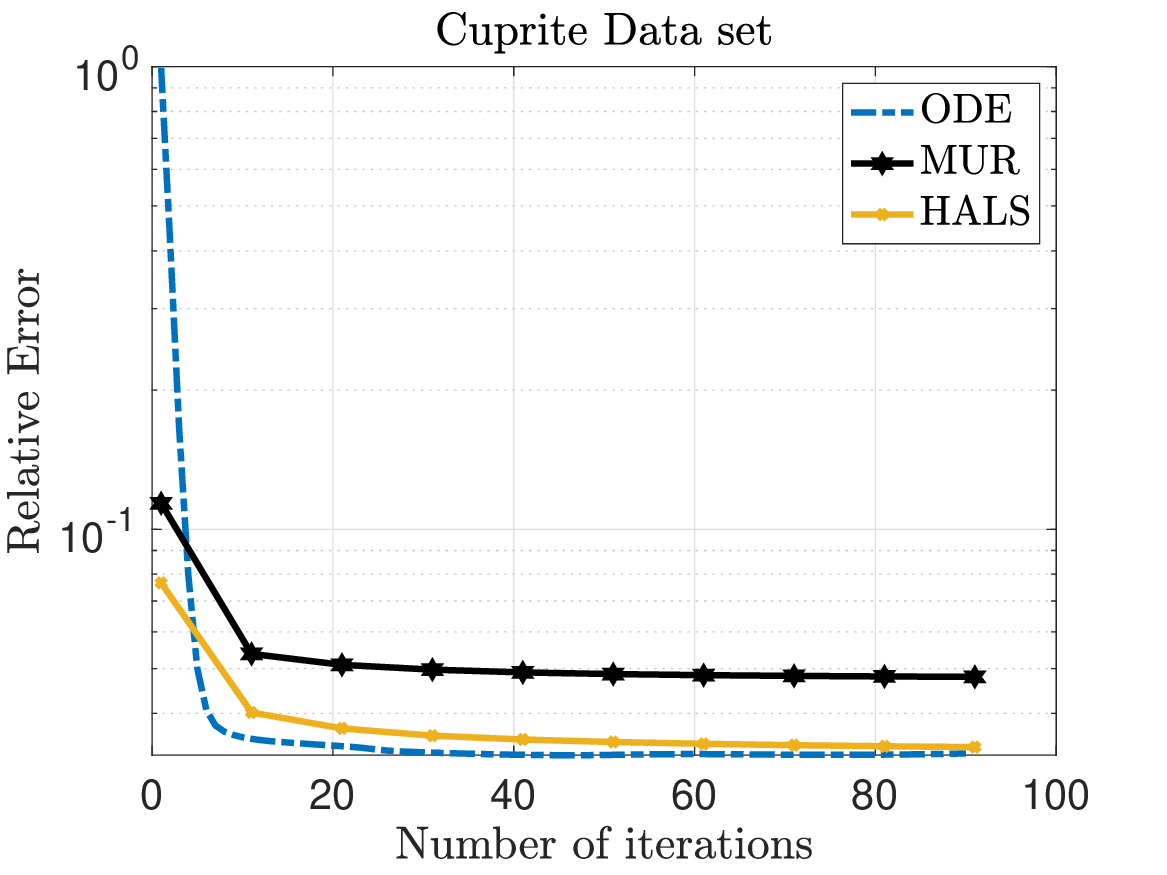}
\includegraphics[width=0.45\columnwidth]{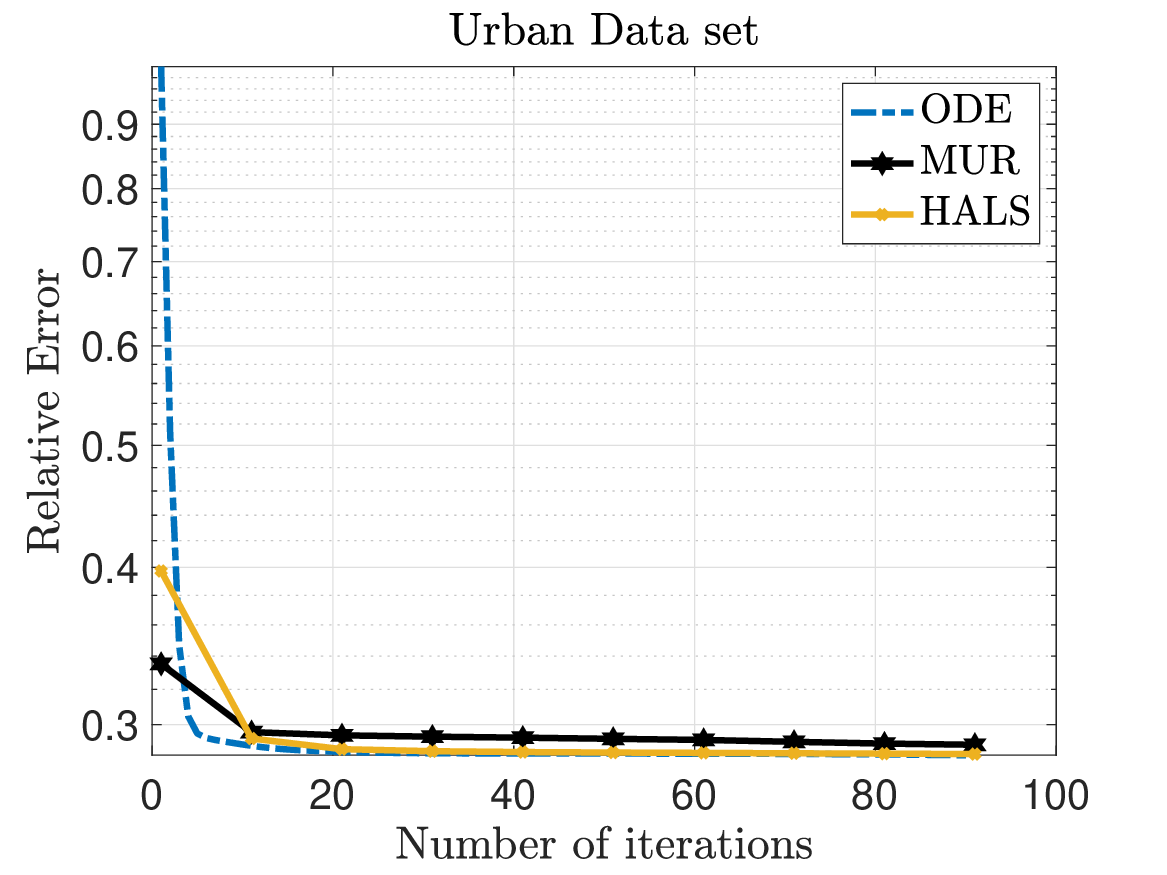}\\
\includegraphics[width=0.45\columnwidth]{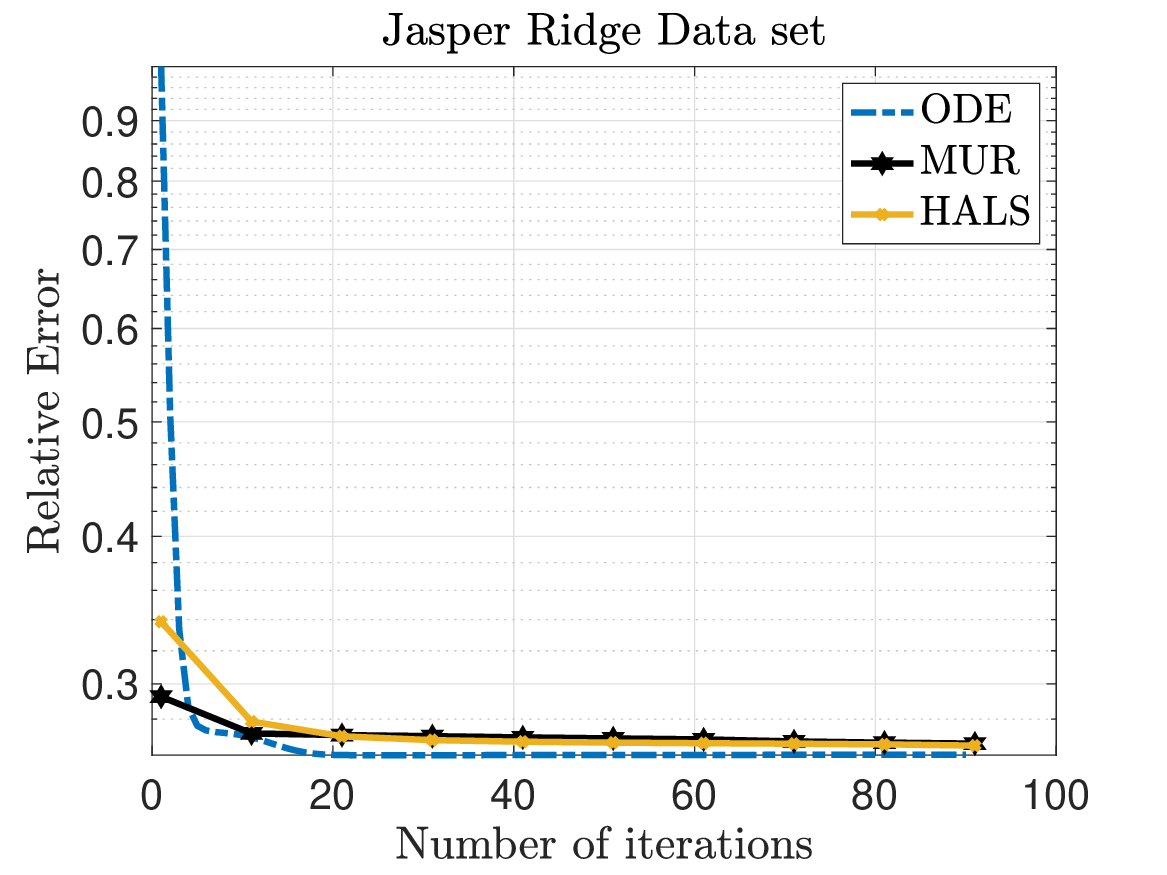}
\caption{\small{ Evolution of the median relative errors found among the 10 random  initializations, of the benchmarked algorithms:
(Top Left) for the Cuprite dataset, (Top Right) for the Urban data set, and (Bottom) for the Jasper Ridge dataset.}}\label{fig:HSI}
\end{center}
\end{figure}

\section{Conclusion}\label{Sec:Con}
In this paper, we have formulated the nonnegative Canonical Polyadic Decomposition (CPD) problem as a Collaborative neurodynamic Optimization (CNO) problem. {Multiple recurrent neural networks communicate with each other using the particle swarm optimization (PSO) method to identify the best candidate, leading to a potential global minimum.} The experiments conducted involve both random and real-world datasets. The results obtained have confirmed the effectiveness of the proposed approach. However, when dealing with large data tensors, obtaining the stable solution of the dynamic system \ref{CON_1} may require a significant time-step, making it computationally expensive. To overcome this challenge, one possible solution is to divide the time-step into smaller intervals and integrate the solutions obtained from these smaller time-steps in parallel. This approach is an area of interest for future research. Additionally, our future work aims to develop collaborative neurodynamic models for other tensor decompositions, such as Tucker decomposition, Tensor Train, Tensor Ring, block term decomposition, and their constrained variants. We also plan to explore nonnegative tensor factorization with Kullback-Leibler (K-L) divergence and Alpha-Beta divergences.

\section{Acknowledgement} The authors would like to thank the editor and two reviewers for their constructive comments and suggestions that allowed us to significantly improve the quality of the paper.

\bibliographystyle{siamplain}
\bibliography{references}

\section{appendix}
{\bf Proof of Theorem \ref{Lem2}.} To prove this theorem, we should consider the following three cases: 
\begin{itemize}
    \item Case I. If $l_i\leq{\bf q}_{k,{\bf a}}^i\leq u_i,\,l_i\leq{\bf q}_{k,{\bf b}}^i\leq u_i$ and $l_i\leq{\bf q}_{k,{\bf c}}^i\leq u_i$, then since $P_{\Omega}({\bf q}_{k,{\bf a}}^i)={\bf q}_{k,{\bf a}}^i,\,P_{\Omega}({\bf q}_{k,{\bf b}}^i)={\bf q}_{k,{\bf b}}^i,\,P_{\Omega}({\bf q}_{k,{\bf c}}^i)={\bf q}_{k,{\bf c}}^i$, we have
\begin{eqnarray}\label{CaseI}
\left\{\begin{aligned}
\ab^{i}_{k+1}&=&\ab^{i}_k-\gamma_k(\nabla_{\bf a} f)^{i},\\
\bb^{i}_{k+1}&=&\bb^{i}_k-\gamma_k(\nabla_{\bf b} f)^{i},\\
\cc^{i}_{k+1}&=&\cc^{i}_k-\gamma_k(\nabla_{\bf c} f)^{i},
\end{aligned}\right.
\end{eqnarray} 
and this proves the first part. In view of \eqref{CaseI}, the DTPNN reaches its equilibrium points iff $\nabla_{\tau}f=0,\,\tau={\bf a},\,{\bf b},\,{\bf c}.$

    \item  Case II. If ${\bf q}_{k,{\bf a}}^i>u_i,\,{\bf q}_{k,{\bf b}}^i>u_i,\,{\bf q}_{k,{\bf c}}^i>u_i$, then since $l_i\leq{\bf a}_k^i,{\bf b}_k^i,{\bf c}_k^i\leq u_i,$ we have
    $\nabla_{\bf a}f< 0,\,\nabla_{\bf b}f< 0,\,\nabla_{\bf c}f< 0,$
    and
    \begin{eqnarray}\label{C}
\left\{\begin{aligned}
\ab_{k+1}^{i}&=&\ab_k^{i}+\lambda_k(-\ab^{i}_k+u_i)\geq 0,\\
\bb^{i}_{k+1}&=&\bb^{i}_k+\lambda_k(-\bb^{i}_k+u_i)\geq 0,\\
\cc^{i}_{k+1}&=&\cc^{i}_k+\lambda_k(-\cc^{i}_k+u_i)\geq 0.\end{aligned}\right.
\end{eqnarray} 
So, we can consider a function $\alpha_i({\bf \tau}_k)\geq 0,\,\tau={\bf a},\,{\bf b},\,{\bf c}$ such that    
 \begin{eqnarray}\label{C_2}
\left\{\begin{aligned}
\ab_{k+1}^{i}&=&\ab_k^{i}+\alpha_i({\bf a}_k)\geq 0,\\
\bb^{i}_{k+1}&=&\bb^{i}_k+\alpha_i({\bf b}_k)\geq 0,\\
\cc^{i}_{k+1}&=&\cc^{i}_k+\alpha_i({\bf c}_k)\geq 0.\end{aligned}\right.
\end{eqnarray}    
and this also proves the second part. In view of \eqref{C_2}, the DTPNN reaches its equilibrium points iff $\alpha_i({\bf c}_k)$.

    \item Case III. If ${\bf q}_{k,{\bf a}}^i<l_i,\,{\bf q}_{k,{\bf b}}^i<l_i,\,{\bf q}_{k,{\bf c}}^i<l_i$, then since $l_i\leq{\bf a}_k^i,{\bf b}_k^i,{\bf c}_k^i\leq u_i,$ we have
    $\nabla_{\bf a}f> 0,\,\nabla_{\bf b}f> 0,\,\nabla_{\bf c}f> 0,$
    and
    \begin{eqnarray}
\left\{\begin{aligned}
\ab_{k+1}^{i}&=&\ab_k^{i}+\lambda_k(-\ab^{i}_k+l_i)\leq 0,\\
\bb^{i}_{k+1}&=&\bb^{i}_k+\lambda_k(-\bb^{i}_k+l_i)\leq 0,\\
\cc^{i}_{k+1}&=&\cc^{i}_k+\lambda_k(-\cc^{i}_k+l_i)\leq 0.\end{aligned}\right.
\end{eqnarray} 
So, we can consider a function $\beta_i({\bf \tau}_k)\geq 0,\,\tau={\bf a},\,{\bf b},\,{\bf c}$ such that    
 \begin{eqnarray}\label{C_3}
\left\{\begin{aligned}
\ab_{k+1}^{i}&=&\ab_k^{i}+\beta_i({\bf a}_k)\leq 0,\\
\bb^{i}_{k+1}&=&\bb^{i}_k+\beta_i({\bf b}_k)\leq 0,\\
\cc^{i}_{k+1}&=&\cc^{i}_k+\beta_i({\bf c}_k)\leq 0.\end{aligned}\right.
\end{eqnarray}    
and this also proves the third part. In view of \eqref{C_3}, the DTPNN reaches its equilibrium points iff $\alpha_i({\bf c}_k)=0$.  
\end{itemize}
{\bf Proof of Theorem \ref{Thm_pro}.}
The proof of this theorem is quite similar to the ones given in \cite{che2018nonnegative} and \cite{li2020discrete}, but for the paper to be self-contained, we present them. Let us consider 
\begin{eqnarray}\label{Fomu_1}
\nonumber
\begin{bmatrix}
{\bf a}_{k+1}\\
{\bf b}_{k+1}\\
{\bf c}_{k+1}
\end{bmatrix}
-
\begin{bmatrix}
\widehat{{\bf a}}\\
\widehat{{\bf b}}\\
\widehat{{\bf c}}
\end{bmatrix}
&=&\begin{bmatrix}
{\bf a}_{k}
\\{\bf b}_{k}
\\{\bf c}_{k}
\end{bmatrix}
-
\begin{bmatrix}
\widehat{{\bf a}}\\
\widehat{{\bf b}}\\
\widehat{{\bf c}}
\end{bmatrix}\\
&+&\lambda_k
\begin{bmatrix}
-{\bf a}_{k}+[{\bf a}_{k}-\nabla_{{\bf a}}f]_{+}
\\-{\bf b}_{k}+[{\bf b}_{k}-\nabla_{{\bf b}}f]_{+}
-{\bf c}_{k}+[{\bf c}_{k}-\nabla_{{\bf c}}f]_{+}
\end{bmatrix},
\end{eqnarray}
and take the Euclidean norm on both sides of \eqref{Fomu_1} as
\begin{eqnarray}
\norm{\begin{bmatrix}
{\bf a}_{k+1}\\
{\bf b}_{k+1}\\
{\bf c}_{k+1}
\end{bmatrix}
-
\begin{bmatrix}
\nonumber
\widehat{{\bf a}}\\
\widehat{{\bf b}}\\
\widehat{{\bf c}}
\end{bmatrix}}_2^2=
\nonumber
\norm{\begin{bmatrix}
{\bf a}_{k}
\\{\bf b}_{k}
\\{\bf c}_{k}
\end{bmatrix}-
\begin{bmatrix}
\widehat{{\bf a}}\\
\widehat{{\bf b}}\\
\widehat{{\bf c}}
\end{bmatrix}
+\lambda_k
\begin{bmatrix}
-{\bf a}_{k}+[{\bf a}_{k}-\nabla_{{\bf a}}f]_{+}
\\-{\bf b}_{k}+[{\bf b}_{k}-\nabla_{{\bf b}}f]_{+}\\
-{\bf c}_{k}+[{\bf c}_{k}-\nabla_{{\bf c}}f]_{+}
\end{bmatrix}
}_2^2=\hspace*{-1cm}\\
\nonumber
\norm{\begin{bmatrix}
{\bf a}_{k}
\\{\bf b}_{k}
\\{\bf c}_{k}
\end{bmatrix}-
\begin{bmatrix}
\widehat{{\bf a}}\\
\widehat{{\bf b}}\\
\widehat{{\bf c}}
\end{bmatrix}}_2^2+
\norm{\lambda_k
\begin{bmatrix}
-{\bf a}_{k}+[{\bf a}_{k}-\nabla_{{\bf a}}f]_{+}
\\-{\bf b}_{k}+[{\bf b}_{k}-\nabla_{{\bf b}}f]_{+}\\
-{\bf c}_{k}+({\bf c}_{k}-\nabla_{{\bf c}}f]_{+}
\end{bmatrix}}_2^2
\hspace*{-1.2cm}\\+2\lambda_k
\begin{bmatrix}
{\bf a}_{k} -{\widehat{{\bf a}}}
\\{\bf b}_{k} -{\widehat{{\bf b}}}
\\{\bf c}_{k} -{\widehat{{\bf c}}}
\end{bmatrix}^T\begin{bmatrix}
-{\bf a}_{k}+({\bf a}_{k}-\nabla_{{\bf a}}f]_{+}
\\-{\bf b}_{k}+({\bf b}_{k}-\nabla_{{\bf b}}f]_{+}\\
-{\bf c}_{k}+({\bf c}_{k}-\nabla_{{\bf c}}f]_{+}
\end{bmatrix}.\hspace*{-.2cm}
\end{eqnarray}
From Lemma \ref{Lem1}, we have
\begin{eqnarray}\label{Formula_2}
\nonumber
\left(
\begin{bmatrix}
({\bf a}_{k}-\nabla_{{\bf a}}f]_{+}
\\({\bf b}_{k}-\nabla_{{\bf b}}f]_{+}\\
({\bf c}_{k}-\nabla_{{\bf c}}f]_{+}
\end{bmatrix}-
\begin{bmatrix}
({\bf a}_{k}-\nabla_{{\bf a}}f)
\\({\bf b}_{k}-\nabla_{{\bf b}}f)\\
({\bf c}_{k}-\nabla_{{\bf c}}f)
\end{bmatrix}\right)^T\times\hspace*{1.5cm}\\
\left(\begin{bmatrix}
({\bf a}_{k}-\nabla_{{\bf a}}f]_{+}
\\({\bf b}_{k}-\nabla_{{\bf b}}f]_{+}\\
({\bf c}_{k}-\nabla_{{\bf c}}f]_{+}
\end{bmatrix}
-
\begin{bmatrix}
{\bf a}_{k}\\
{\bf b}_{k}\\
{\bf c}_{k}
\end{bmatrix}
+\begin{bmatrix}
{\bf a}_{k}\\
{\bf b}_{k}\\
{\bf c}_{k}
\end{bmatrix}
-\begin{bmatrix}
\widehat{\bf a}\\
\widehat{\bf b}\\
\widehat{\bf c}
\end{bmatrix}\right)\leq 0,
\end{eqnarray}
and using some straightforward computations, from \eqref{Formula_2} we can get
{\small
\begin{eqnarray}\label{Formul_3}
\nonumber
\hspace{-4cm}\begin{bmatrix}
{\bf a}_{k}-\widehat{\bf a}\\
{\bf b}_{k}-\widehat{\bf b}\\
{\bf c}_{k}-\widehat{\bf c}
\end{bmatrix}^T
\begin{bmatrix}
-{\bf a}_{k}+({\bf a}_{k}-\nabla_{{\bf a}}f]_{+}
\\-{\bf b}_{k}+({\bf b}_{k}-\nabla_{{\bf b}}f]_{+}\\
-{\bf c}_{k}+({\bf c}_{k}-\nabla_{{\bf c}}f]_{+}
\end{bmatrix}\leq\\
-\begin{bmatrix}
\nabla_{{\bf a}}f\\
\nabla_{{\bf b}}f\\
\nabla_{{\bf c}}f
\end{bmatrix}^T
\begin{bmatrix}
-{\bf a}_{k}+({\bf a}_{k}-\nabla_{{\bf a}}f]_{+}
\\-{\bf b}_{k}+({\bf b}_{k}-\nabla_{{\bf b}}f]_{+}\\
-{\bf c}_{k}+({\bf c}_{k}-\nabla_{{\bf c}}f]_{+}
\end{bmatrix}
-\begin{bmatrix}
\nabla_{{\bf a}}f\\
\nabla_{{\bf b}}f\\
\nabla_{{\bf c}}f
\end{bmatrix}^T
\begin{bmatrix}
{\bf a}_{k}-\widehat{\bf a}
\\{\bf b}_{k}-\widehat{\bf b}\\
{\bf c}_{k}-\widehat{\bf c}
\end{bmatrix}
\nonumber
\\
\nonumber
-\norm{
\begin{bmatrix}
-{\bf a}_{k}+({\bf a}_{k}-\nabla_{{\bf a}}f]_{+}
\\-{\bf b}_{k}+({\bf b}_{k}-\nabla_{{\bf b}}f]_{+}\\
-{\bf c}_{k}+({\bf c}_{k}-\nabla_{{\bf c}}f]_{+}
\end{bmatrix}}_2^2\leq\\
\nonumber
-\begin{bmatrix}
\nabla_{{\bf a}}f\\
\nabla_{{\bf b}}f\\
\nabla_{{\bf c}}f
\end{bmatrix}^T
\begin{bmatrix}
-{\bf a}_{k}+({\bf a}_{k}-\nabla_{{\bf a}}f]_{+}
\\-{\bf b}_{k}+({\bf b}_{k}-\nabla_{{\bf b}}f]_{+}\\
-{\bf c}_{k}+({\bf c}_{k}-\nabla_{{\bf c}}f]_{+}
\end{bmatrix}
-\\
\norm{
\begin{bmatrix}
-{\bf a}_{k}+({\bf a}_{k}-\nabla_{{\bf a}}f]_{+}
\\-{\bf b}_{k}+({\bf b}_{k}-\nabla_{{\bf b}}f]_{+}\\
-{\bf c}_{k}+({\bf c}_{k}-\nabla_{{\bf c}}f]_{+}
\end{bmatrix}}_2^2,
\end{eqnarray}}
where $\nabla_{\bf a}f({\bf a}_k-\widehat{\bf a})\geq 0,\,\nabla_{\bf b}f({\bf b}_k-\widehat{\bf b})\geq 0,\,\nabla_{\bf c}f({\bf c}_k-\widehat{\bf c})\geq 0$, have been used in derivation of the last inequality in \eqref{Formul_3}. From Lemma \ref{Lem2} and \eqref{Formul_3}, we have
\begin{eqnarray}
\nonumber
\begin{bmatrix}
{\bf a}_{k}-\widehat{\bf a}\\
{\bf b}_{k}-\widehat{\bf b}\\
{\bf c}_{k}-\widehat{\bf c}
\end{bmatrix}^T
\begin{bmatrix}
-{\bf a}_{k}+({\bf a}_{k}-\nabla_{{\bf a}}f]_{+}
\\-{\bf b}_{k}+({\bf b}_{k}-\nabla_{{\bf b}}f]_{+}\\
-{\bf c}_{k}+({\bf c}_{k}-\nabla_{{\bf c}}f]_{+}
\end{bmatrix}\leq\\
\frac{1}{\lambda_k}\norm{
\begin{bmatrix}
\gamma_k^T\nabla_{\bf a}f\\
\gamma_k^T\nabla_{\bf b}f\\
\gamma_k^T\nabla_{\bf c}f
\end{bmatrix}
}
-\norm{
\begin{bmatrix}
-{\bf a}_{k}+({\bf a}_{k}-\nabla_{{\bf a}}f]_{+}
\\-{\bf b}_{k}+({\bf b}_{k}-\nabla_{{\bf b}}f]_{+}\\
-{\bf c}_{k}+({\bf c}_{k}-\nabla_{{\bf c}}f]_{+}
\end{bmatrix}}_2^2,
\end{eqnarray}
Let us now define the Lyapunov function defined as follows
\[
L({\bf a},{\bf b},{\bf c})=\norm{
\begin{bmatrix}
{\bf a}-\widehat{\bf a}\\
{\bf b}-\widehat{\bf b}\\
{\bf c}-\widehat{\bf c}
\end{bmatrix}
}_2^2,
\]
and for the stability of the iterative method, we need to satisfy 
\begin{equation}\label{Lyapu}
L({\bf a}_{k+1},{\bf b}_{k+1},{\bf c}_{k+1})-L({\bf a}_{k},{\bf b}_{k},{\bf c}_{k})\leq 0.
\end{equation}
It can be seen 
\begin{eqnarray}\label{Lyapu_2}
\nonumber
L({\bf a}_{k+1},{\bf b}_{k+1},{\bf c}_{k+1})-L({\bf a}_{k},{\bf b}_{k},{\bf c}_{k})\leq
2\norm{
\begin{bmatrix}
\gamma_k^T\nabla_{\bf a}f\\
\gamma_k^T\nabla_{\bf b}f\\
\gamma_k^T\nabla_{\bf c}f
\end{bmatrix}
}_2^2\\
+(\lambda_k^2-\lambda_k)\norm{
\begin{bmatrix}
-{\bf a}_{k}+({\bf a}_{k}-\nabla_{{\bf a}}f]_{+}
\\-{\bf b}_{k}+({\bf b}_{k}-\nabla_{{\bf b}}f]_{+}\\
-{\bf c}_{k}+({\bf c}_{k}-\nabla_{{\bf c}}f]_{+}
\end{bmatrix}}_2^2
\end{eqnarray}
Combining \eqref{Lyapu} and \eqref{Lyapu_2}, we can easily find an interval for $\lambda_k$ as 
\begin{equation}\label{interv}
\max(0,1-\sqrt{c})\leq\lambda_k\leq(1+\sqrt{c}),
\end{equation}
where
\begin{equation*}
c=\left(1-2\norm{\begin{bmatrix}\gamma^T_k\nabla_{\bf a} f\\\gamma^T_k\nabla_{\bf b} f\\\gamma^T_k\nabla_{\bf c} f\end{bmatrix}}_2^2\right)\left(\norm{\begin{bmatrix}{\bf a}_k-({\bf a}_k-\nabla_{\bf a}f]_{+}\\{\bf b}_k-({\bf b}_k-\nabla_{\bf b}f]_{+}\\{\bf c}_k-({\bf c}_k-\nabla_{\bf c}f]_{+}\end{bmatrix}}_2^2\right)^{-1}.
\end{equation*}
So, the DTPNN \eqref{CON_2}-\eqref{CON_223} is stable if \eqref{interv} is holds. According to  LaSalle’s invariance principle \cite{la1976stability}, it is known that the states of the DTPNN \eqref{CON_2}-\eqref{CON_223} converge to the largest invariant subset defined as
\begin{equation*}\label{Equ}
\Gamma=\{{\bf a}_k,{\bf b}_k,{\bf c}_k\in\mathbb{R}^n|L({\bf a}_{k+1},{\bf b}_{k+1},{\bf c}_{k+1})-L({\bf a}_{k},{\bf b}_{k},{\bf c}_{k})=0\}.
\end{equation*}
A triplet $({\bf a}_k,{\bf b}_k,{\bf c}_k)$ belongs to the set $\Gamma$ if 
\[
\norm{
\begin{bmatrix}
\gamma_k^T\nabla_{\bf a}f\\
\gamma_k^T\nabla_{\bf b}f\\
\gamma_k^T\nabla_{\bf c}f
\end{bmatrix}
}_2^2=\norm{
\begin{bmatrix}
-{\bf a}_{k}+({\bf a}_{k}-\nabla_{{\bf a}}f]_{+}
\\-{\bf b}_{k}+({\bf b}_{k}-\nabla_{{\bf b}}f]_{+}\\
-{\bf c}_{k}+({\bf c}_{k}-\nabla_{{\bf c}}f]_{+}
\end{bmatrix}}_2^2=0,
\]
and according to Lemma \ref{Lem2}, this means that the DTPNN \eqref{CON_2}-\eqref{CON_223} converges to its equilibrium points. 
Now, from \eqref{Lyapu}, we have
\[
\norm{
\begin{bmatrix}
{\bf a}_{k+1}-\widehat{\bf a}\\
{\bf b}_{k+1}-\widehat{\bf b}\\
{\bf c}_{k+1}-\widehat{\bf c}
\end{bmatrix}
}_2^2\leq\norm{
\begin{bmatrix}
{\bf a}_{k}-\widehat{\bf a}\\
{\bf b}_{k}-\widehat{\bf b}\\
{\bf c}_{k}-\widehat{\bf c}
\end{bmatrix}
}_2^2
\leq \cdots 
\norm{
\begin{bmatrix}
{\bf a}_{0}-\widehat{\bf a}\\
{\bf b}_{0}-\widehat{\bf b}\\
{\bf c}_{0}-\widehat{\bf c}
\end{bmatrix}
}_2^2,
\]
which shows that the triple $({\bf a}_k,{\bf b}_k,{\bf c}_k)$ is bounded in $\mathbb{R}^n$ and therefore there exists a subset $k_t$ for which 
\[
\lim_{k \to +\infty} 
\norm{
\begin{bmatrix}
{\bf a}_{k_t}-\widehat{\bf a}\\
{\bf b}_{k_t}-\widehat{\bf b}\\
{\bf c}_{k_t}-\widehat{\bf c}
\end{bmatrix}
}_2^2=0,
\]
and
\[
\lim_{k \to +\infty} 
\begin{bmatrix}
{\bf a}_{k_t}\\
{\bf b}_{k_t}\\
{\bf c}_{k_t}
\end{bmatrix}=\begin{bmatrix}
\widehat{\bf a}\\
\widehat{\bf b}\\
\widehat{\bf c}
\end{bmatrix}.
\]
This shows that the iterative method is also convergent. 

\end{document}

\maketitle

% \section{A detailed example}

% Here we include some equations and theorem-like environments to show
% how these are labeled in a supplement and can be referenced from the
% main text.
% Consider the following equation:
% \begin{equation}
%   \label{eq:suppa}
%   a^2 + b^2 = c^2.
% \end{equation}
% You can also reference equations such as \cref{eq:matrices,eq:bb} 
% from the main article in this supplement.

% \lipsum[100-101]

% \begin{theorem}
%   An example theorem.
% \end{theorem}

% \lipsum[102]
 
% \begin{lemma}
%   An example lemma.
% \end{lemma}

% \lipsum[103-105]

% Here is an example citation: \cite{KoMa14}.

% \section[Proof of Thm]{Proof of \cref{thm:bigthm}}
% \label{sec:proof}
% \lipsum[106-112]

% \section{Additional experimental results}
% \Cref{tab:foo} shows additional
% supporting evidence. 

% \begin{table}[htbp]
% {\footnotesize
%   \caption{Example table}  \label{tab:foo}
% \begin{center}
%   \begin{tabular}{|c|c|c|} \hline
%    Species & \bf Mean & \bf Std.~Dev. \\ \hline
%     1 & 3.4 & 1.2 \\
%     2 & 5.4 & 0.6 \\ \hline
%   \end{tabular}
% \end{center}
% }
% \end{table}

\bibliographystyle{siamplain}
\bibliography{references}